\documentclass[a4paper,11pt,oneside]{article}
\usepackage[english]{babel}
\usepackage[T1]{fontenc}
\usepackage[utf8]{inputenc}
\usepackage{amsthm}
\usepackage{bbm}
\usepackage{amsmath}
\usepackage{amssymb}
\usepackage{indentfirst}
\usepackage{fancyhdr}
\usepackage{amsthm}
\usepackage{graphicx}
\usepackage{pdfpages}
\usepackage{esint}
\usepackage{url}

\newcommand{\nc}{\normalcolor}

\numberwithin{equation}{section}

\theoremstyle{plain}
\newtheorem{thm}{Theorem}[section]
\newtheorem{cor}[thm]{Corollary}
\newtheorem{lem}[thm]{Lemma}
\newtheorem{prop}[thm]{Proposition}

\newtheorem{dfn}[thm]{Definition}

\usepackage{xcolor}
\definecolor{custom-blue}{RGB}{0,99,166}
\usepackage{hyperref}
\hypersetup{colorlinks=true, allcolors=custom-blue}
\usepackage{geometry}
\allowdisplaybreaks[4]

\begin{document}


\title{{Some obstacle problems for partially hinged plates\\ and related optimization issues}}

\author{Elvise Berchio\thanks{Dipartimento di Scienze Matematiche, Politecnico di Torino, Corso Duca degli Abruzzi, 24, 10129, Torino, Italy. Email: elvise.berchio@polito.it.} \and 
Filomena Feo \thanks{Dipartimento di Ingegneria, Università degli Studi di Napoli ``Parthenope'',
Centro Direzionale Isola C4, Napoli, 80143, Italy. Email: filomena.feo@uniparthenope.it.}
\and 
Antonio Giuseppe Grimaldi \thanks{Dipartimento di Ingegneria, Università degli Studi di Napoli ``Parthenope'',
Centro Direzionale Isola C4, Napoli, 80143, Italy. Email: antoniogiuseppe.grimaldi@collaboratore.uniparthenope.it.}}

\date{}

\maketitle

\begin{abstract}

We study optimization problems for partially hinged rectangular plates, modeling bridge roadways, in the presence of real and artificial obstacles. Real obstacles represent structural constraints to avoid, while artificial ones are introduced to enhance stability. For the former, aiming to prevent collisions, we set up a worst-case optimization problem in which we minimize the amplitude of oscillations with respect to the density distribution; for the latter, aiming to improve the torsional stability, we minimize, with respect to the obstacles, the maximum of a gap function quantifying the displacement between the long edges of the plate. For both problems, existence results are provided, along with a discussion about qualitative properties of optimal density distributions and obstacles.

\end{abstract}


 \section{Introduction}
Recent years have seen growing interest in the study of partially hinged rectangular plates, due to their applicability in mathematical models of footbridges and suspension bridges. Indeed, the roadway of a bridge can be modeled by a long, thin plate, hinged at the short edges and free at the long edges, see \cite{antunes}-\cite{BeFeGa}, \cite{FeGa}, \cite{bookgaz}.
Expanding upon this model, in this work we consider an obstacle problem for partially hinged rectangular plates, and related optimization issues, where the obstacles considered are either \emph{real} or \emph{artificial}.

By \emph{real} obstacles we mean structural elements of the bridge positioned above or below the roadway, elements that the roadway, i.e., the plate in our model, should not collide with. Examples include multi-level bridges or roadway coverings. A fundamental issue for the related mathematical models analysis is to possibly suggest ways to prevent collisions. In the paper we face this issue by setting up a worst-case optimization problem aiming to reduce the amplitude of the oscillations in the deck by varying the density distribution within the plate.

The second class of obstacle problems considered in the work involves \emph{artificial} obstacles, i.e., obstacles, such as long metal guides (applied, e.g., along the long edges of the plate), intentionally introduced into the structure to enhance its stability. In this case collisions are allowed. It is well known that one of the main causes of instability in suspension bridges is that they are susceptible to torsional instability, which can lead to significant performance issues; for further details we refer to \cite{BeFeGa} and the monograph \cite{bookgaz}. To mitigate this phenomenon, several previous papers were devoted to studying the effect on the structure of modifying the shape of the plate (\cite{BBG}) or the distribution of materials composing it (\cite{berchio,befa}). In order to measure the gain in stability achieved when performing the above-mentioned changes, a function was introduced in \cite{berchio} named \textit{gap function}, measuring the displacement between the long edges, see formula \eqref{gap2bis} below. Then a suitable worst-case optimization problem was set up to improve the torsional stability of the plate.  In our analysis we properly modify this problem by minimizing the maximum of the gap function with respect to suitable families of obstacles. We point out that, while the modification of the shape or the density distribution of the plate must obviously precede its construction, the insertion of artificial obstacles can be thought of as a remedy applicable afterward to existing structures whose performance in terms of stability one wishes to improve.

It is worth mentioning that, while the cost functional we minimize is entirely different, the second class of problems considered in the paper shares similarities with optimal control problems in which the obstacle itself plays the role of the control. These problems were first studied in \cite{Adams}, and subsequently addressed by various authors, both in the second-order and higher-order settings, for linear as well as nonlinear variational inequalities (see, e.g., \cite{Adams1, dimitri, GN} and references therein).

Moreover, we observe that, even if both optimization problems considered in the work are inspired by the one proposed in \cite{berchio}, crucial differences emerge in the analysis. In particular, when real obstacles are involved, the function being optimized is different, whereas in the case of artificial obstacles, the distinction lies in the set over which the optimization is performed. Besides, the model considered in \cite{berchio} did not account for obstacles, and their introduction produces nontrivial modifications in the mathematical framework with the consequent failure of several arguments exploited in the obstacles-free case. For example, the rescaling of an admissible test function may no longer satisfy the admissibility conditions imposed by the obstacles, therefore the strategy of constructing solutions for rescaled data by linearity fails, in general. As a matter of fact, in the proposed analysis these difficulties add to those usually arising when dealing with higher order variational (in)equalities, such as the lack of weak maximum and comparison principles and the failure of reflection and truncation arguments. These technical obstructions explain why, in the higher order case, the literature on obstacles problems and free boundary problems is comparatively less developed than that in the second order one, see e.g., \cite{Al,Brezis, Caf1,Caf2,Caf3,Da, Da2,dipierro,Frehse1,Frehse2,grunau,Lov,NovOk2} and references therein. While most previous studies focus on clamped or fully hinged plate models, to the best of our knowledge, this work is the first to address obstacle problems involving partially hinged plates. Nevertheless, we observe that regularity issues, central in the existing literature, are not explored in this paper. Indeed, even if mathematically interesting, they are are not of central importance to the analysis presented here, particularly in view of the applications discussed above.
\par

\medskip
\par
The paper is organized as follows: in Section \ref{obst_general}, we introduce the mathematical framework of the obstacle problem for partially hinged plates studied in the article, including its variational formulation and sufficient conditions for empty contact sets. Section \ref{real} is devoted to real obstacles, where we set up a worst case optimization problem to prevent collisions by modifying the plate density distribution (see problem \eqref{min max max}). In Section \ref{artificial} we consider the case of artificial obstacles, focusing on minimizing the torsional response via the gap function (see problem \eqref{min max max2bis}). For both problems, existence results are provided, along with qualitative properties of optimal density distributions and obstacles, including symmetry considerations and explicit constructions based on the Green function representation, see Sections \ref{Green} and \ref{Green2}.

\section{The obstacle problem for partially hinged plates}\label{obst_general}

In what follows, up to scaling, we assume that the plate $\Omega$ has length $\pi$ and width $2l \ll \pi$ so that $\Omega = (0,\pi) \times (-l,l) \subset \mathbb{R}^2$. When the plate $\Omega$ is homogeneous, according to the Kirchhoff-Love theory (see \cite[Chapter 1]{book}), the energy $\mathbb{E}$ of the vertical deformation $u$ of $\Omega$ subject to a load $f$ may be computed through the functional
\begin{equation}\label{energia}
    \mathbb{E}(u, \Omega):= \int_\Omega \left( \dfrac{(\Delta u)^2}{2} +(1-\sigma)(u^2_{xy}-u_{xx}u_{yy})-fu\right)  dxdy,
\end{equation}
where $\sigma \in (0,1)$ is the Poisson ratio. Under this condition on $\sigma$ the quadratic part of the energy $\mathbb{E}$ is positive. Since in our model the plate is assumed to be hinged at the short edges, the natural functional space where to set the analysis is
$$H^2_{*}(\Omega):= \left\{  v \in H^2(\Omega) \ :  \ v=0 \ \text{on} \ \{ 0, \pi \} \times (-l,l) \right\}$$
which is a Hilbert space when endowed with the scalar product
\begin{equation*}
     (u,v)_{H^2_{*}(\Omega)}:=  \int_\Omega \left( \Delta u \Delta v +(1-\sigma)(2u_{xy}v_{xy}-u_{xx}v_{yy}-u_{yy}v_{xx})\right)  dxdy,
\end{equation*}
 and the associated norm $\Vert u \Vert^2_{ {H^2_{*}(\Omega)}}:= (u,u)_{H^2_{*}(\Omega)}$ for all $\sigma \in (0,1)$. This property is proved in \cite[Lemma 4.1]{FeGa} where authors show that the norm $\Vert \cdot \Vert^2_{ {H^2_{*}(\Omega)}}$ \nc is equivalent to the standard norm in $H^2(\Omega)$, i.e.\ $\Vert u \Vert^2_{ {H^2(\Omega)}} = \Vert u \Vert^2_{ {L^2(\Omega)}}+\Vert D^2u \Vert^2_{ {L^2(\Omega)}}$.

We are interested in a partially hinged plate restricted to remain between two prescribed obstacles, therefore the functional $\mathbb{E}$ should be minimized on the  closed convex subset of $H^2_{*}(\Omega)$:
\begin{equation}
    H^2_{*,\psi_{-},\psi_{+}}(\Omega):= \left\{  v \in H^2_*(\Omega) \ : \ \psi_{-} \le v \le \psi_{+} \ \text{in} \  \overline \Omega \right\}, \notag
\end{equation}
where $\psi_{-}(x) \le 0 \le \psi_{+}(x)$, with $\psi_{-} \not\equiv 0$ and $\psi_{+} \not\equiv 0$, are two prescribed continuous functions defined on $\overline \Omega$.
We remark that, since $\Omega$ is a planar domain, the space $H^2(\Omega)$  is compactly embedded in $C^0(\overline{\Omega})$ and then the conditions $v=0 \ \text{on} \ \{ 0, \pi \} \times (-l,l)$ and $\psi_{-} \le v \le \psi_{+}$ in $\overline \Omega$ are pointwise satisfied.
Concerning the assumptions on the load $f$, the functional  $\mathbb{E}$ is well-defined if $f$ belongs to $L^1(\Omega)$. Otherwise, we set $(C^0(\overline{\Omega}))'$ the dual space of $C^0(\overline{\Omega})$ and we denote by $\langle \cdot, \cdot \rangle$ the duality product. By  Riesz theorem the dual $(C^0(\Omega))'$ is isometric to the space of Radon measures. If $\mu$ is the measure associated to $f\in (C^0(\overline{\Omega}))'$, then we have to replace $\int_\Omega u \, d\mu$ with $\langle f,u\rangle$ in \eqref{energia} and $ \Vert f \Vert_{(C^0(\overline{\Omega}))'}=\mu(\Omega)$.

For all $f \in (C^0(\overline{\Omega}))'$ the minimizer $u=u_{f} \in H^2_{*,\psi_{-},\psi_{+}}(\Omega)$ of $\mathbb{E}$ satisfies the variational inequality
 \begin{equation}\label{var0}
     (u,\varphi-u)_{H^2_{*}(\Omega)}  \ge \langle f, \varphi-u \rangle \qquad \forall \varphi \in H^2_{*,\psi_{-},\psi_{+}}(\Omega)\,.
 \end{equation}
In a natural way, we associate $u$ with the following contact sets:
\begin{align}\label{sets}
 \Omega_{-}:=\{(x,y)\in \overline{\Omega}: u(x,y) = \psi_{-}(x,y)\} \quad \text{and} \quad \Omega_{+}:=\{(x,y)\in \overline{\Omega} : u(x,y) = \psi_{+}(x,y)\}\,.
\end{align}
Since $u\in C^0(\overline{\Omega})$, the sets $\Omega_{-}$ and $\Omega_{-}$ are closed and if we also assume that  $\psi_-<0 < \psi_+$ they are disjoint. When both $\Omega_{-}$ and $\Omega_{+}$ are empty, namely $\psi_{-} < u < \psi_{+}$ in $\overline{\Omega}$, by taking as test function in \eqref{var0}, $u\pm \varepsilon \phi$ with $\phi\in H^2_*(\Omega)$ and $\varepsilon>0$ sufficiently small, it follows that $u$  satisfies the partially hinged plate problem:
\begin{equation}\label{loadpb0weak}
	(u,\phi)_{H^2_*(\Omega)} =\langle f, \phi \rangle  \qquad\forall \phi\in H^2_*(\Omega)\,
	\end{equation}
    which in strong form reads
\begin{equation}\label{loadpb0}
\begin{cases}
	\Delta^2 u=f
 &  \text{in } \Omega\,, \\
	u=u_{xx}=0 & \text{on } \{0,\pi\}\times ]-l,l[\,, \\
	u_{yy}+\sigma u_{xx}=u_{yyy}+(2-\sigma)u_{xxy}=0 & \text{on } ]0,\pi[\times \{-l,l\}\,.
\end{cases}
\end{equation}
 See, \cite{FeGa} and \cite{book} for a detailed explanation of the model and for the derivation of the boundary conditions in \eqref{loadpb0}.

\par

Let $G_p$ denote the Green function of the biharmonic operator on $\Omega$, under partially hinged boundary conditions, namely the solution to \eqref{loadpb0weak} with $\delta_{p}$ instead of $f$, where $\delta_p$ is the Dirac delta with mass concentrated at $p\in \overline \Omega$. By showing that $G_p$ is strictly positive in $\overline \Omega$ (see Section \ref{Green} for more details), in \cite[Theorem 2.2]{BF2} it was proved that problem \eqref{loadpb0weak} satisfies the \textit{positivity preserving property} which means that if $f\in L^2(\Omega)$ and $u\in H^2_*(\Omega)$ is the solution of \eqref{loadpb0weak}, then the following implication holds:  
    \begin{equation}\label{ppp}
		f\geq 0,\,\, f\not\equiv 0  \text{ in } \Omega\quad  \Rightarrow \quad u>0  \text{ in } (0,\pi)\times[-l,l]\,.
		\end{equation}
    Thanks to \eqref{ppp}, we deduce a sufficient condition for having empty contact sets:
    \begin{prop}\label{empty}
   Let $ \mathcal{F}=\{f\in L^{\infty}(\Omega)\,:\,\|f\|_{\infty}\leq 1\}$. Furthermore, let $G_p$ be the Green function of problem \eqref{loadpb0} (given explicitly in formula \eqref{green3} below) and let $\psi_{\pm} \in C^0(\overline \Omega)$ be such that 
\begin{equation}
|\psi_{\pm}(x,y)|>  \int_{\Omega} G_p(x,y)\,dp \qquad \text{in } \overline \Omega. \label{green.eq}
\end{equation}    
      Then, the contact sets \eqref{sets} of problem \eqref{var0}, with $f\in \mathcal{F}$ and obstacles $\psi_-$ and $\psi_+$, are empty, namely the unique minimizer of $\mathbb{E}$ as defined in \eqref{energia} satisfies $\psi_{-}< u < \psi_{+}$ in $\overline \Omega$.

    \end{prop}

    \begin{proof}
    Denote by $z \in H^2_*(\Omega)$, the unique solution to the partially hinged plate problem \eqref{loadpb0weak} with $f\equiv 1$. From \eqref{ppp}, we infer that $z>0$ in $(0,\pi)\times[-l,l]$. Let $u \in H^2_*(\Omega)$ be the unique solution to the partially hinged plate problem \eqref{loadpb0weak} with $f\in \mathcal{F}$.  Clearly, $z \pm u \in H^2_*(\Omega)$ satisfy \eqref{loadpb0weak} with load $g_{\pm}=1\pm f$. By definition of $\mathcal{F}$, $g_{\pm} \geq 0$ in $\overline \Omega$, hence, by \eqref{ppp}, 
 \begin{equation}
 z \pm u \geq 0 \qquad \text{in } \overline \Omega.  \label{zpmu}
\end{equation}    
Noticing that $ z(x,y)= \int_{\Omega} G_p(x,y)\,dp$ for all $ (x,y)\in \Omega$, \eqref{green.eq} and \eqref{zpmu} give that
     $$\psi_{-}< -z \leq u \leq z < \psi_{+} \qquad \text{in } \overline \Omega.$$  Furthermore, for all $\varphi \in H^2_{*,\psi_{-},\psi_{+}}(\Omega)$, taking as test function in \eqref{loadpb0weak} $\varphi-u$, it is readily seen that $u$ satisfies \eqref{var0} with the equality. This completes the proof.
    \end{proof}

\section{A worst-case optimization problem to avoid collisions}\label{real}
In this section we assume that the obstacles $\psi_{\pm}$ (as defined in Section \ref{obst_general}) are fixed. Expanding on the model proposed in \cite{berchio}, we examine two strategies to prevent contact with the obstacles: increasing the cost associated with the bending energy or reducing the effect of the applied force. In applications these changes can be interpreted, respectively, as reinforcing suitable regions of the plate or changing its surface density distribution. 

More precisely, in the following we denote by $D \subset \Omega$ an open region and $D^c:=\Omega\setminus D$. Furthermore, we assume that $D$ belongs to a certain class of sets $\mathcal{D}$, while $f$ belongs to some space $\mathcal{F}$ of admissible forcing terms. Both classes will be specified later on. Let $0<\alpha< 1< \beta$, we modify the original energy \eqref{energia} into the following two ways:
\begin{equation}\label{energia1}
    \mathbb{E}_1(u, \Omega):= \int_\Omega \left[ \left( \beta \chi_D+\alpha \chi_{D^c} \right)\left( \dfrac{(\Delta u)^2}{2} +(1-\sigma)(u^2_{xy}-u_{xx}u_{yy})\right)-fu \right] dxdy
\end{equation}
and
\begin{equation}\label{energia2}
    \mathbb{E}_2(u, \Omega):= \int_\Omega \left[  \dfrac{(\Delta u)^2}{2} +(1-\sigma)(u^2_{xy}-u_{xx}u_{yy})-\left( \beta \chi_D+\alpha \chi_{D^c} \right)fu \right] dxdy,
\end{equation}
where $\chi_D$ and $ \chi_{D^c}$ are the characteristic functions of $D$ and $D^c$. The choice of the two-step constant function $\beta \chi_D+\alpha \chi_{D^c}$ is suggested by the classical theory for composite membranes which has been recently extended to plates, see, e.g., \cite{Chanillo,cox,CV2} and the survey paper \cite{vecchi}. The constants $\alpha, \beta$ stand for the densities of two different materials, located in different parts of the plate. In order to make the comparison between different choices of $D,\alpha, \beta$ consistent, when $\alpha< 1< \beta$ we assume that $|D|=|\Omega | \dfrac{1-\alpha}{\beta-\alpha}$ so that $\int_{\Omega}\left( \beta \chi_D+\alpha \chi_{D^c} \right) dxdy=|\Omega|$ for all $D\in \mathcal{D}$.  The quadratic part of the functionals \eqref{energia1} and \eqref{energia2} are positive as happens for \eqref{energia} and will be minimized on the space  $H^2_{*,\psi_{-},\psi_{+}}(\Omega)$.
To deal with \eqref{energia1} for any open set $D \subset \Omega$ we introduce the following  bilinear form
 \begin{equation*}
     (u,v)_{D}:=  \int_D \left( \Delta u \Delta v +(1-\sigma)(2u_{xy}v_{xy}-u_{xx}v_{yy}-u_{yy}v_{xx})\right)  dxdy.
 \end{equation*}
Clearly, $(u,v)_{\Omega}=(u,v)_{H^2_{*}(\Omega)}$. Then, for all $f \in (C^0(\overline{\Omega}))'$ the minimizer $u_{f,D}\in H^2_{*,\psi_{-},\psi_{+}}(\Omega)$ of $\mathbb{E}_1$ satisfies the variational inequality
 \begin{equation}\label{var1}
    \alpha (u_{f,D},\varphi-u_{f,D})_{H^2_{*}(\Omega)}+(\beta -\alpha) (u_{f,D},\varphi-u_{f,D})_{D} \ge \langle f, \varphi-u_{f,D} \rangle \qquad \forall \varphi \in H^2_{*,\psi_{-},\psi_{+}}(\Omega).
 \end{equation}
  For what concerns $\mathbb{E}_2$, it is well defined for any $f \in L^p(\Omega)$ with $p\geq 1$, but not in general for any $f\in (C^0(\overline{\Omega}))'$. Moreover, the minimizer satisfies the variational inequality
 \begin{equation}\label{var2}
     (u_{f,D},\varphi-u_{f,D})_{H^2_{*}(\Omega)} \ge  \int_\Omega  \left( \beta \chi_D+\alpha \chi_{D^c} \right)f (\varphi-u_{f,D}) \ dxdy \qquad \forall \varphi \in H^2_{*,\psi_{-},\psi_{+}}(\Omega).
 \end{equation}

Both $\mathbb{E}_1$ and $\mathbb{E}_2$ admit a unique minimizer in $H^2_{*,\psi_{-},\psi_{+}}(\Omega)$. Indeed, by \cite{FeGa}, the bilinear form $(u,\varphi-u)_{H^2_{*}(\Omega)}$ is continuous and coercive. Furthermore, both the functionals on the right-hand side of \eqref{var1} and \eqref{var2} are linear and continuous, besides $H^2_{*,\psi_{-},\psi_{+}}(\Omega)$ is a closed, convex subset of $H^2_{*}(\Omega)$.


Finally, we set up a worst case optimization problem aiming to prevent collisions between the plate and the obstacles. For $(f,D)  \in \mathcal{F} \times \mathcal{D}$, we denote by $u_{f,D} \in H^2_{*,\psi_{-},\psi_{+}}(\Omega)$ the minimizer of $\mathbb{E}_1$ or $\mathbb{E}_2$. We notice that $u_{f,D}$ depends on $\alpha$ and $\beta$ as well but we will not indicate this dependence to keep the notation as simple as possible. Given $D\in\mathcal{D}$, we first look for the \textit{worst} $f \in \mathcal{F}$ yielding the maximal $L^{\infty}$ norm of $u_{f,D}$:
\begin{equation}\label{max max}
    \mathcal{A}^\infty_{D}:= \max_{f \in \mathcal{F}} \mathcal{A}^\infty_{f,D} = \max_{f \in \mathcal{F}} \max_{(x,y) \in \overline \Omega} |u_{f,D}(x,y)|\,.
\end{equation}
Then, we search for the \textit{best} $D \in \mathcal{D}$ minimizing the $L^{\infty}$ norm in the worst scenario:
\begin{equation}\label{min max max}
    \mathcal{A}^\infty:= \min_{D \in \mathcal{D}} \mathcal{A}^\infty_{D} = \min_{D \in \mathcal{D}} \max_{f \in \mathcal{F}} \max_{(x,y) \in \overline \Omega} |u_{f,D}(x,y)|.
\end{equation}
In Section \ref{ex} we provide some classes $\mathcal{F}$ and  $\mathcal{D}$ in which \eqref{max max} and \eqref{min max max} admit a solution, while in Section \ref{Green}, when $\mathbb{E}=\mathbb{E}_2$, we exploit the Green function representation formula to suggest possible locations of best reinforcements, namely of the set $D$ in $\Omega$.

\subsection{Existence results}\label{ex}

We start by showing that $  \mathcal{A}^\infty_{D}$ defined in \eqref{max max} is well defined for a suitable choice of $\mathcal{F}$.


\begin{thm}\label{Th max max}
For any open set $D \subset \Omega$ and every $p \in (1,+\infty]$, the problems
\begin{equation}\label{Pb1}
    \mathcal{A}^\infty_{D}:= \max_{f \in \mathcal{F}} \mathcal{A}^\infty_{f,D} \quad \text{with} \quad  \mathcal{F}:= \left\{ f \in (C^0(\overline{\Omega}))' \ : \ \Vert f \Vert_{
(C^0(\overline{\Omega}))'} \leq  1  \right\} \quad (\text{for} \ \mathbb{E}_1),
\end{equation}
and
\begin{equation}\label{Pb2}
    \mathcal{A}^\infty_{D}:= \max_{f \in \mathcal{F}} \mathcal{A}^\infty_{f,D} \quad \text{with} \quad  \mathcal{F}:= \left\{ f \in L^p(\Omega) \ : \ \Vert f \Vert_{L^p(\Omega)} \leq  1  \right\} \quad (\text{for both} \ \mathbb{E}_1 \ \text{and} \ \mathbb{E}_2 ),
\end{equation}
admit a solution. Furthermore, in both cases, if $-\psi_-\equiv \psi_+=\psi$, and $f$ is a maximizer of $ \mathcal{A}^\infty_{D}$, then $-f$ is a maximizer as well.
\end{thm}

\medskip

In order to prove Theorem \ref{Th max max} we first show the continuity of the map $\mathcal{A}^\infty_{f,D}$ defined in \eqref{max max}.

\begin{prop}\label{cont}
   Let $p \in (1,+\infty]$ and let $\mathcal{D}$ be a class of open subdomains of $\Omega$ closed with respect to the $L^1$ topology. Then the maps
   $$(f,D) \in (C^0(\overline{\Omega}))' \times \mathcal{D} \mapsto \mathcal{A}^\infty_{f,D} \in [0,+\infty) \quad (\text{for} \ \mathbb{E}_1)$$
   $$(f,D) \in L^p(\Omega) \times \mathcal{D} \mapsto \mathcal{A}^\infty_{f,D} \in [0,+\infty) \quad (\text{for both} \ \mathbb{E}_1 \ \text{and} \ \mathbb{E}_2 ) $$
   are sequentially continuous, when $(C^0(\overline{\Omega}))'$ and $L^p(\Omega)$ are endowed with the weak* topology and $\mathcal{D}$ is endowed with the $L^1$ topology.
   \end{prop}


\proof

We first consider the case of energy $\mathbb{E}_1$. Let $\{(f_n,D_n) \}_n \subset \mathcal{F} \times \mathcal{D}$ be such that $(f_n,D_n) \to (f,D)$ as $n \to + \infty$, that is $f_n \rightharpoonup f$ weakly* in $\mathcal{F}$ ( i.e. $f_n \rightharpoonup f$ weakly* in $(C^0(\overline{\Omega}))'$ or $f_n \rightharpoonup f$ in $L^p(\Omega)$ if $1<p<\infty$  or  $f_n \rightharpoonup f$ weakly* in $L^\infty(\Omega)$) \nc
and $\chi_{D_n} \to \chi_D$ strongly in $L^1(\Omega)$ as $n \to + \infty$. For simplicity of notation, we denote by $u=u_{f,D}$ and $u_n=u_{f_n,D_n}$, for all $n \in \mathbb{N}$, the corresponding solutions of \eqref{var1}. In particular, for every $n \in \mathbb{N}$, it holds
 \begin{equation}\label{eq1}
     \alpha(u_n,\varphi-u_n)_{H^2_{*}(\Omega)}+(\beta -\alpha) (u_n,\varphi-u_n)_{D_n} \ge \langle f_n, \varphi-u_n \rangle \qquad \forall \varphi \in  H^2_{*,\psi_{-},\psi_{+}}(\Omega).
 \end{equation}
Since the zero constant function belongs to the class of admissible functions $ H^2_{*,\psi_{-},\psi_{+}}(\Omega)$, we have
 \begin{equation}
     (u_n,-u_n)_{H^2_{*}(\Omega)}+(\beta -\alpha) (u_n,-u_n)_{D_n} \ge \langle f_n, -u_n \rangle \notag
 \end{equation}
and we deduce that for a suitable constant $C>0,$ independent of $u_n$, there holds
\begin{equation*}
    \Vert u_n \Vert^2_{H^2_*(\Omega)} \le C \ \Vert f_n \Vert_{ \mathcal{F}} \ \Vert u_n \Vert_{H^2_*(\Omega)}. 
\end{equation*}
Since $f_n \rightharpoonup f$ weakly* in $ \mathcal{F}$, the sequence $\{ f_n \}_n$ is bounded in $\mathcal{F} $. Therefore, $\{ u_n \}_n$ is bounded in $H^2_*(\Omega)$, and then there exists $\overline{u} \in H^2_*(\Omega)$ such that, up to a subsequence, $u_n \rightharpoonup \overline{u}$ weakly in $H^2_*(\Omega)$. Moreover, since $H^2_{*,\psi_{-},\psi_{+}}(\Omega)$ is closed and convex, we have that $\overline{u} \in H^2_{*,\psi_{-},\psi_{+}}(\Omega)$.
\\
 By the weak convergence of $u_n$ to $\overline{u}$ in $H^2_{*}(\Omega)$, $\Vert \overline{u} \Vert_{H^2_*(\Omega)}\leq \Vert u_n \Vert_{H^2_*(\Omega)}+o(1)$ as $n\rightarrow +\infty$ and the first term on the left-hand side of \eqref{eq1} can be estimated as follows
 \begin{equation}\label{eq02}
   (u_n,\varphi-u_n)_{H^2_{*}(\Omega)}  = (u_n,\varphi)_{H^2_{*}(\Omega)} -\Vert u_n \Vert_{H^2_*(\Omega)}^2 \le (\overline{u},\varphi-\overline{u})_{H^2_{*} (\Omega)}+o(1)\quad \text{as } n\rightarrow +\infty\,.
\end{equation}
Next we consider the second term on the left-hand side of \eqref{eq1}. To this aim, we first notice that
\begin{equation}\label{uDn} 
(\overline{u},\overline{u})_{D \setminus D_n}^{1/2} \rightarrow 0 \quad \text{as } n\rightarrow +\infty\,
 \end{equation}
 and 
 \begin{equation}\label{unDn}
(\overline{u},\overline{u})_D\leq (u_n,u_n)_{D_n}+o(1)\quad \text{as } n\rightarrow +\infty\,.
 \end{equation}
 The limit in \eqref{uDn} follows by exploiting the strong convergence of $\chi_{D_n}$ to $\chi_D$ in $L^1(\Omega)$ while \eqref{unDn} follows by combining the weak convergence of $u_n$ to $\overline{u}$ in $H^2_{*}(\Omega)$, the boundedness of $\{ u_n \}_n$, and \eqref{uDn}. Indeed, we have
 \begin{align*}
 (\overline{u},\overline{u})_D=(u_n,\overline{u})_D+o(1)&= (u_n,\overline{u})_{D \cap D_n}+(u_n,\overline{u})_{D \setminus D_n}+o(1)\\
 &\leq (u_n,u_n)_{D}^{1/2}(\overline{u},\overline{u})_{D}^{1/2}+(u_n,u_n)_{D }^{1/2}(\overline{u},\overline{u})_{D \setminus D_n}^{1/2}+o(1)\\
 &\leq (u_n,u_n)_{D}^{1/2}(\overline{u},\overline{u})_{D}^{1/2}+o(1)
 \quad \text{as } n\rightarrow +\infty\,.
 \end{align*}
Finally, from \eqref{unDn}, we get
\begin{align}\label{final}
     (u_n,\varphi-u_n)_{D_n} \leq &   (u_n,\varphi)_{D_n}- (\overline{u},\overline{u})_D+o(1) \notag\\
     = & \  (u_n,\varphi)_{D}  +   (u_n,\varphi)_{D_n\setminus D}-   (u_n,\varphi)_{D \setminus D_n}- (\overline{u},\overline{u})_D+o(1) \notag \\
     = & \  (\overline{u},\varphi-\overline{u} )_{D}  +o(1),
\end{align}
where again we have exploited the strong convergence of $\chi_{D_n}$ to $\chi_D$ in $L^1(\Omega)$ and the boundedness of the sequence $\{ u_n \}_n$ to get
\begin{equation*}
    |(\varphi,u_n)_{D_n\setminus D}| \le C \Vert \varphi \Vert_{H^2_*(D_n \setminus D)} \to 0 \ \text{as} \ n \to + \infty
\end{equation*}
and
\begin{equation*}
    |(\varphi,u_n)_{D\setminus D_n}| \le C \Vert \varphi \Vert_{H^2_*(D \setminus D_n)} \to 0 \ \text{as} \ n \to + \infty.
\end{equation*}
 From the compactness of the embedding $H^2(\Omega) \subset C^0(\overline{\Omega})$, we deduce that $u_n \to \overline{u}$ in $C^0(\overline{\Omega})$. This, together with the fact that $f_n \rightharpoonup f$ weakly* in $\mathcal{F}$, gives that
 \begin{equation}\label{eq8}
     \langle f_n, \varphi-u_n \rangle \to \langle f, \varphi-\overline{u} \rangle \ \text{as} \ n \to + \infty.
 \end{equation}
Passing to the limit as $n \to + \infty$ in \eqref{eq1} and using \eqref{eq02}, \eqref{final} and \eqref{eq8}, we get
 \begin{equation*}
     \alpha(\overline{u},\varphi-\overline{u})_{H^2_{*}(\Omega)}+(\beta -\alpha) (\overline{u},\varphi-\overline{u})_{D} \ge \langle f, \varphi-\overline{u} \rangle \qquad \forall \varphi \in H^2_{*, \psi_{-},\psi_{+}}(\Omega).
 \end{equation*}
 Thus, we can conclude that $\overline{u}=u$ in $\Omega$, by the uniqueness of solutions.
\\
Since $u_n \to u$ in $C^0(\overline{\Omega})$, we find that $\mathcal{A}_{f_n,D_n} \to \mathcal{A}_{f,D}$ uniformly as $n \to + \infty$ in $[0,\pi]$, i.e.\ $\mathcal{A}^\infty_{f_n,D_n} \to \mathcal{A}^\infty_{f,D}$ as $n \to + \infty$.

A similar proof works for the energy $\mathbb{E}_2$.
 \endproof

\medskip

\medskip

\proof [Proof of Theorem \ref{Th max max}] Fix $D \subset \mathcal{D}$. We consider the two cases.
\par \medskip \par
\textit{ $\bullet$ CASE 1:} $\mathcal{F}:= \left\{ f \in (C^0(\overline{\Omega}))' \ : \ \Vert f \Vert_{
(C^0(\overline{\Omega}))'} \leq  1  \right\}$ and the energy $\mathbb{E}_1$.    Let $\{ f_n \}_n \subset (C^0(\overline{\Omega}))'$ be a maximizing sequence for \eqref{Pb1} such that $\|f_n\|_{ (C^0(\overline{\Omega}))'}\leq 1$. Since $\{ f_n \}_n$ is bounded in $(C^0(\overline{\Omega}))'$, 
there exists $\overline{f} \in (C^0(\overline{\Omega}))'$ such that, up to a subsequence, $f_n \rightharpoonup^* \overline{f}$ in $(C^0(\overline{\Omega}))'$. Then, by Proposition \ref{cont}, we get
$$\mathcal{A}^\infty_{\overline{f},D}=\mathcal{A}^\infty_{D}  .$$ On the other hand, by weak lower semicontinuity of the norm, we have
$$ \Vert \overline{f} \Vert_{(C^0(\overline{\Omega}))'} \le \liminf_{n \to + \infty} \Vert f_n \Vert_{(C^0(\overline{\Omega}))'}\leq 1.$$
Hence, $\overline{f} \in \mathcal{F}$ and solves problem \eqref{Pb1}.

\par \medskip \par
\textit{$\bullet$ CASE 2:}  $\mathcal{F}:= \left\{ f \in L^p(\Omega) \ : \ \Vert f \Vert_{L^p(\Omega)} \leq  1  \right\}$ and the energy $\mathbb{E}_1$ or $\mathbb{E}_2$.
Let $\{ f_n \}_n \subset L^p(\Omega)$ be a maximizing sequence for \eqref{Pb2} such that $\|f_n\|_{L^p}\leq 1$. Since $\{ f_n \}_n$ is bounded in $ L^p(\Omega)$, there exists $\overline{f} \in  L^p(\Omega)$ such that, up to a subsequence, $f_n \rightharpoonup \overline{f}$ in $L^p(\Omega)$ for $1<p<\infty$ and $f_n \rightharpoonup^* \overline{f}$ in $L^\infty(\Omega)$.  By Proposition \ref{cont}, we get
$$\mathcal{A}^\infty_{\overline{f},D}=\mathcal{A}^\infty_{D}  .$$
Furthermore, as in the previous case, we have that $\overline{f} \in \mathcal{F}$, hence it solves problem \eqref{Pb2}. \par
Finally, we show that if $-\psi_-\equiv \psi_+\equiv \psi$, and $f$ is a maximizer of $ \mathcal{A}^\infty_{D}$, then $-f$ is a maximizer as well. We first observe that $\varphi  \in H^2_{*,-\psi,\psi}(\Omega)$ if and only if $-\varphi \in H^2_{*,-\psi,\psi}(\Omega)$.  Consequently, if $u_{f,D} \in H^2_{*,-\psi,\psi}(\Omega)$ satisfies either \eqref{var1} or \eqref{var2}, then $-u_{f,D}\in H^2_{*,-\psi,\psi}(\Omega)$ and it satisfies the corresponding inequality with $-f$ in place of $f$, thus proving the claim.   
\endproof

\medskip

In order to show that problem \eqref{min max max} is well-defined as well, we consider some classes of admissible domains originally introduced in \cite{berchio} and which may have applicative interest.

\medskip

\begin{dfn}\label{rinforzi ammissibili} \cite[Definition 3.1]{berchio} 
\smallskip
(a) \emph{Cross-type reinforcements}:
$$\mathcal C :=\Big\{D\subset\Omega: \, D=\Big(\bigcup_{i=1}^N ]x_i-\mu,x_i+\mu[\times ]-l,l[\Big)\cup\Big(\bigcup_{j=1}^M ]0,\pi[\times ]y_j-\varepsilon,y_j+\varepsilon[ \Big)\Big\}\,,$$
where for $N,M\in\mathbb N$, $\mu\in ]0,\pi/2N[$, $\varepsilon\in]0,l/M[$, $x_i\in [\mu, \pi-\mu]$ for $i=1,\dots, N$ with
$x_{i+1}-x_i>2\mu$ for $i\le N-1$, and $y_j\in[-l+\varepsilon,l-\varepsilon]$ for $j=1,\dots, M$ with $y_{j+1}-y_j>2\varepsilon$ for $j\le M-1$.

 \smallskip
(b)  \emph{Tiles of rectangular shapes}:
$$\mathcal T :=\Big\{D\subset\Omega: \, D=\bigcup_{i=1}^N R^i, \text{$R^i\subset\Omega$ is an open rectangle with inradius  $\geq\varepsilon$
}\Big\}\,,$$
where $N\in\mathbb N$ and $\varepsilon\in ]0,l[$.

 \smallskip
(c) \emph{Networks of bounded length}:
$$\mathcal N :=\big\{D\subset\Omega:\, \text{$D=\Sigma^\varepsilon$ where $\Sigma\subset \overline\Omega$ is closed, connected, $\mathcal H^1(\Sigma)\leq L$}\big\}\,,$$
where $\varepsilon\in]0, l[$ and $L>0$. $\mathcal H^1$ denotes the one-dimensional Hausdorff measure of a set, and $\Sigma^\varepsilon$ represents the $\varepsilon$-tubular neighborhood of $\Sigma$, namely the set of points in $\Omega$ at distance to $\Sigma$ less than $\varepsilon$.
 \smallskip

(d) \emph{Lipschitz trusses}:
$$\mathcal L :=\big\{D\subset \Omega: \, \text{$D$ open with the inner $\varepsilon$-cone property}\big\}\,,$$
where $\varepsilon\in ]0,l[$ and the inner $\varepsilon$-cone property means that at every point $x$ of the boundary $\partial D$ there is some truncated cone from $x$ with an opening angle $\varepsilon$ and radius $\varepsilon$ inside $D$.
\end{dfn}

Some of these classes are monotone with respect to set inclusion, namely $\mathcal C\subset \mathcal T\subset \mathcal L$,
for suitable choices of the parameters $N, \varepsilon, \mu, L$. We show that problem \eqref{min max max} is well-defined if the class $\mathcal{D}$ is one of those introduced in the above definition.

\medskip

\begin{thm}\label{minthm}
    Let $\alpha< 1< \beta$ and let $\mathcal C, \mathcal T, \mathcal N$ and $\mathcal L$ be the classes of sets introduced in Definition \ref{rinforzi ammissibili}. The minimization problem
    $$\mathcal{A}^\infty:= \min_{D \in \mathcal{D}} \mathcal{A}^\infty_{D}$$
    with
    \begin{equation*}
        \mathcal{D}:=\left\{ D \in \mathcal C \,(\text{or } \mathcal T \text{or }  \mathcal N \text{or }  \mathcal L) \text{ and }   \ |D| = |\Omega |\, \dfrac{1-\alpha}{\beta-\alpha}  \right\}
    \end{equation*}
    admits a solution.
\end{thm}

\begin{proof}
We first notice that by its definition, the continuity proved in Proposition \ref{cont} and the existence of a worst force given in Theorem \ref{Th max max}, the functional $D\mapsto \mathcal{A}_D^\infty$ is lower-semicontinuous with respect to the
$L^1$-convergence of sets. Indeed, if $D_n \to D$ in $L^1(\Omega)$ as $n \to + \infty$ and $\overline f\in \mathcal{F}$ is the maximizer of $\mathcal{A}_D^\infty$ given by Theorem \ref{Th max max}, then
$$\mathcal{A}_D^\infty=\mathcal{A}_{\overline f,D}^\infty=\lim_{n \to + \infty}\mathcal{A}_{\overline f,D_n}^\infty\leq \liminf_{n \to + \infty}\mathcal{A}_{D_n}^\infty\,.$$
 Once this noticed, the existence of a solution to problem \eqref{min max max} follows from the Direct Method in the Calculus of Variations for every class of admissible sets $\mathcal{D}$ that is compact with respect to the $L^1$-convergence of sets. In particular, this holds if $\mathcal D$ is contained in one of the classes of sets $\mathcal C, \mathcal T, \mathcal N$ and $\mathcal L$, introduced in Definition \ref{rinforzi ammissibili} and for which this compactness issue was established in \cite[Theorem 3.2]{berchio}.
\end{proof}

\subsection{Hints about the location of the best reinforcements}\label{Green}
We have proved that for $\mathcal{F}$ and $\mathcal{D}$ as given in Theorems \ref{Th max max} and \ref{minthm}, respectively, problem \eqref{min max max} admits a solution. Clearly, from an applied perspective, it would be interesting to apriori determine qualitative properties of the set(s) $D$ that attain the minimum (and, in turn, allow to prevent collisions with the obstacles). In this section we exploit the representation formula for the solution of the partially hinged plate problem \eqref{loadpb0weak} proven in \cite{BF2}, to provide some hints about the best location of the set $D$ within the plate, i.e., where to put the heavier (with density $\beta$) and, in turn, the lighter (with density $\alpha$) materials within the plate.

As in Section \ref{obst_general}, we denote by $G_p$ the Green function of the biharmonic operator on $\Omega$, under partially hinged boundary conditions. Then, the solution of \eqref{loadpb0weak} with $f\in L^2(\Omega)$ writes
	\begin{equation}\label{espl}
    u(x,y)= \int_{\Omega} G_p(x,y) f(p)\,dp \qquad \forall (x,y)\in\Omega \,. 
    \end{equation}
The Fourier expansion of $G_p$ was computed in \cite[Theorem 2.1]{BF2} and reads
	\begin{equation}\label{green3}
	G_p(x,y)=\sum_{m=1}^{+\infty} \dfrac{1}{2\pi}\dfrac{\phi_m(y,\eta)}{m^3}\sin(m\xi)\,\sin(mx) \qquad \forall p=(\xi,\eta) \text{ and }(x,y)\in\overline \Omega\,.
	\end{equation}
	For all $m\in \mathbb{N^+}$, the coefficients $\phi_m(y,\eta)$ are defined as follows
	 \begin{equation}\label{soluzdirac2}
		\begin{split}
		\phi_m(y,\eta):=
		e^{-ml}\bigg[&\cosh(m\eta)\bigg(\dfrac{\overline \zeta(my,ml)}{F(ml)}+ml\dfrac{\overline \psi(my,ml) }{F(ml)}-m \eta\dfrac{\overline \omega(my,ml)}{\overline F(ml)}\bigg)\\+&\sinh(m \eta)\bigg(\dfrac{\overline\vartheta(my,ml)}{\overline F(ml)}+ml\dfrac{\overline \omega(my,ml) }{\overline F(ml)}-m\eta\dfrac{\overline\psi(my,ml)}{F(ml)}\bigg)\bigg] 
		\\+&(1+m|y-\eta|)e^{-m|y-\eta|}
		\end{split}
		\end{equation}
		with $F, \overline F: (0,+\infty)\rightarrow (0,+\infty)$ and $\overline\zeta, \overline\vartheta, \overline\psi, \overline\omega: \mathbb{R}\times(0,+\infty)\rightarrow \mathbb{R}$ as follows
        \begin{equation*}
F(z):=\dfrac{(3+\sigma)}{2}\sinh(2z)-z (1-\sigma)\,, \quad  \overline F(z):=\dfrac{(3+\sigma)}{2}\sinh(2z)+z (1-\sigma)\,,
\end{equation*}  
	\begin{equation*}
	\begin{split}
	\overline\zeta(r,z):=&\bigg(\dfrac{4}{1-\sigma}-z(1+\sigma)\bigg)\cosh(r)\cosh(z)+\bigg(\dfrac{(1+\sigma)^2}{1-\sigma}+2z\bigg)\cosh(r)\sinh(z)\\&-2r\sinh(r)\cosh(z)+r(1+\sigma)\sinh(r)\sinh(z)\\
	\overline \vartheta(r,z):=&r(1+\sigma)\cosh(r)\cosh(z)-2r\cosh(r)\sinh(z)\\&+\bigg(\dfrac{(1+\sigma)^2}{1-\sigma}+2z\bigg)\sinh(r)\cosh(z)+\bigg(\dfrac{4}{1-\sigma}-z(1+\sigma)\bigg)\sinh(r)\sinh(z)\\
	\overline\psi(r,z):=&\big(2+(1-\sigma)z\big)\cosh(r)\cosh(z)+\big(-(1+\sigma)+z(1-\sigma)\big)\cosh(r)\sinh(z)\\&-r(1-\sigma)\sinh(r)\cosh(z)-r(1-\sigma)\sinh(r)\sinh(z)\\
	\overline \omega(r,z):=&-r(1-\sigma)\cosh(r)\cosh(z)-r(1-\sigma)\cosh(r)\sinh(z)\\&+\big(-(1+\sigma)+z(1-\sigma)\big)\sinh(r)\cosh(z)+\big(2+(1-\sigma)z\big)\sinh(r)\sinh(z)\,.
	\end{split}
	\end{equation*}
	
   By means of several delicate estimates, in \cite[Theorem 2.1]{BF2} it was proved that the $\phi_m$ are strictly positive and strictly decreasing with respect to $m$, i.e.,
\begin{equation}\label{decreasing}
		0<\phi_{m+1}(y,\eta)<\phi_m(y,\eta)\qquad\forall m\in\mathbb{N^+},\forall y,\eta\in[-l,l]\,.
		\end{equation}
    Exploiting the above information, in \cite[Theorem 2.2]{BF2} it was proved that the series defining $G_p$ is uniformly convergent in $\overline \Omega$ and, through iterative estimates, that $G_p$ is strictly positive in $(0,\pi)\times[-l,l]$. In turn, problem \eqref{loadpb0weak} satisfies the positivity preserving property \eqref{ppp}. 
    
    Now we focus on the modified energy $\mathbb{E}_2$ with $\mathcal{F}=\{f\in L^{\infty}(\Omega)\,:\,\|f\|_{\infty}\leq 1\}$ and we apriori assume that the plate does not touch the obstacles; we will come back to this assumption at the end. Then, the unique minimizer $u_{f,D}$ of $\mathbb{E}_2$ satisfies \eqref{espl} with $\left( \beta \chi_D+\alpha \chi_{D^c} \right)f$ instead of $f$. Recalling that $G_p$ is positive, for all $D\in \mathcal{D}$, we deduce that:
        \begin{equation}\label{upperb}
         \mathcal{A}^\infty_{D}\leq  \max_{f \in \mathcal{F}} \max_{(x,y) \in \overline \Omega} |u_{f,D}(x,y)|\leq \max_{(x,y) \in \overline \Omega} \left[\beta\int_{D} G_p(x,y)\,dp+ \alpha \int_{D^c} G_p(x,y) \,dp\right]. 
         \end{equation}
        The bound \eqref{upperb} suggests that one possible way to reduce $\mathcal{A}^\infty_{D}$, is to place the heavier and lighter materials where the map $p\mapsto G_p$ attains its minimum and maximum values, respectively. Due to the complexity of the expression of $G_p$, a complete analytical optimization of the right-hand side of \eqref{upperb} proves to be quite challenging. Nevertheless, preliminary insights can be readily obtained through elementary estimates.
      
By exploiting the elementary inequality $|\sin(m \xi)|<m\sin(\xi )$ for all $\xi \in(0,\pi)$ and all positive integer $m$, and \eqref{decreasing} we get that

$$ \mathcal{A}^\infty_{D}\leq 
 \dfrac{\pi}{12} \max_{(x,y) \in \overline \Omega}
\left[\beta\int_{D} \phi_1(y,\eta)\sin(\xi)\,d\eta \, d \xi+ \alpha \int_{D^c} \phi_1(y,\eta) \sin(\xi) \,d\eta \, d \xi\right]\,.
$$
According to the above estimate, the heavier material $\beta$ has to be located where $\sin(\xi)$ has the minimum value, namely near the short edges. Conversely, the lighter material $\alpha$ has to be located where $\sin(\xi)$ has maximum value, namely near the line $\xi=\frac{\pi}{2}$.  \\ The behavior of $\eta \mapsto \phi_1(y,\eta)$ is more delicate to investigate. Nevertheless, using the fact that $\overline F(l)>F(l)>0$ and the following estimates for all $y\in [-l,l]$: 
\begin{equation}\label{coeff-est}
|\overline\zeta(y,l)|\leq A(l,\sigma)\,, \quad 
|\overline \vartheta(y,l)|\leq A(l,\sigma)\,, \quad 
|\overline\psi(y,l)|\leq B (l,\sigma)\,, \quad 
|\overline \omega(y,l)|\leq B(l,\sigma)
	\end{equation}
 with $A(l,\sigma):=\bigg(\dfrac{4+(1+\sigma)^2}{1-\sigma}+2l(3+\sigma)\bigg)\cosh^2(l)$ and $B(l,\sigma):=\big(3+\sigma+4(1-\sigma)l\big)\cosh^2(l)$, we deduce that
\begin{equation}\label{g}
\phi_1(y,\eta) \leq  C(l,\sigma) (\cosh(\eta)+|\sinh(\eta)|)(1+|\eta|)+1=:g(\eta) \qquad \forall y,\eta\in[-l,l]\,,
\end{equation}
where  
$C(l,\sigma):=\dfrac{\cosh^2(l)}{e^{l} F(l)}\left[\dfrac{4+(1+\sigma)^2}{1-\sigma}+l(13-\sigma) \right]>0\,.$
 \\
Since the function $g$ defined in \eqref{g} is even and strictly increasing in $[0,l]$, the upper bound for $\mathcal{A}^\infty_{D}$ suggests that the heavier material $\beta$ should be located near the midline of the plate, while its maximum value is at $\eta=\pm l$, therefore the lighter material should be located near the long edges. Based on the observations made so far, we conclude that:
{\it
\begin{center}
in order to reduce $\mathcal{A}^\infty_{D}$ and thereby prevent collisions between the plate and the obstacles:
\begin{itemize}
\item the set $D$ should include the short edges and the midline of the plate; 
\item the set $D^c$ should include the line $\xi=\frac{\pi}{2}$ and the long edges of the plate.
\end{itemize}
\end{center}
}
In particular, the above suggestions should be followed when choosing the sets $D$ in the family $\mathcal D$ considered in Theorem \ref{minthm}. \par As already remarked, the above estimates hold under the assumption that the contact sets are empty. Recalling the  compatibility condition $|D|=|\Omega| \frac{1-\alpha}{\beta-\alpha}$, this assumption follows from the above estimates by assuming e.g., that the obstacles functions satisfy:
$$|\psi_{\pm}|>\frac{lg(l)\pi^2}{6} \quad \text{in }\overline{\Omega}\quad  \text{ with $g$ as in \eqref{g}}.$$
 Alternatively, let $z_{\beta}\in H^2_*(\Omega)$ be the unique solution to the partially hinged plate problem \eqref{loadpb0weak} with $f\equiv \beta$.  Since $\|\left( \beta \chi_D+\alpha \chi_{D^c} \right)f\|_{\infty}\leq \beta $ for all $f \in \mathcal{F}$ and $0<\alpha< \beta$, arguing as in the proof of Proposition \ref{empty}, we infer that the contact sets of problem \eqref{var2} (with obstacles $\psi_-$ and $\psi_+$) are empty as soon as $\psi_{\pm} \in C^0(\overline \Omega)$ satisfy $$|\psi_{\pm}|>z_{\beta}=\beta \int_{\Omega} G_p(x,y)\,dp \quad \text{in } \overline \Omega\,.$$\nc
\section{A worst-case optimization problem to improve stability} \label{artificial}

We start by introducing a suitable class of obstacles among which to search for the most effective ones to insert into the structure, in order to minimize the torsional response of the plate. More precisely, we assume that the partially hinged plate is forced to remain between two prescribed obstacles located in a suitable not empty, closed region $\Omega_{\text{O}} \subset \overline\Omega$. Therefore, from now on we slightly modify the notation introduced in Section \ref{obst_general} and we minimize the functional $\mathbb{E}$ given  in \eqref{energia} over the set
 \begin{equation}
    H^2_{*,\psi_{-},\psi_{+}}(\Omega):= \left\{  v \in H^2_*(\Omega) \ : \ \psi_{-} \le v \le \psi_{+} \ \text{in} \ \Omega_{\text{O}}\right\}, \notag
\end{equation}
where $\psi_{+}$ and $\psi_{-}$ belong, respectively, to the following sets:
\begin{equation}
    \text{O}_+(\gamma_+):= \left\{ \psi:\Omega_{\text{O}}\to \mathbb{R} \ : \ \psi \in \mathcal{C}^0(\Omega_{\text{O}}), \ \psi \ge \gamma_+ \ \text{in} \ \Omega_{\text{O}}\right\}\label{O+bis}
\end{equation}
and
\begin{equation}
    \text{O}_-(\gamma_-):= \left\{ \psi:\Omega_{\text{O}}\to \mathbb{R} \ : \ \psi \in \mathcal{C}^0(\Omega_{\text{O}}), \ \psi \le- \gamma_- \ \text{in} \ \Omega_{\text{O}} \right\},\label{O-bis}
\end{equation}
for some $\gamma_+, \gamma_- >0$.
Then, for a given $f \in (C^0(\overline{\Omega}))'$ the minimizer $u_{f,\psi_{-},\psi_{+}}\in H^2_{*,\psi_{-},\psi_{+}}(\Omega) $ of $\mathbb{E}$ satisfies the variational inequality
 \begin{equation}\label{var1bis}
     (u_{f,\psi_{-},\psi_{+}},\varphi-u_{f,\psi_{-},\psi_{+}})_{H^2_{*}(\Omega)} \ge \langle f, \varphi-u_{f,\psi_{-},\psi_{+}} \rangle \qquad \forall \varphi \in H^2_{*,\psi_{-},\psi_{+}}(\Omega)\,.
 \end{equation}
  In the case $\Omega_\text{O}$ is a curve, we refer to \eqref{var1bis} as a \textit{thin obstacle problem} (see, \cite{Schild} and \cite{Schild2}).  
If $\psi_{-} < u_{f,\psi_{-},\psi_{+}} < \psi_{+}$ in $\Omega_\text{O}$, as in the case of obstacles defined over the whole set $\overline \Omega$, by taking as test function in \eqref{var1bis}, $u_{f,\psi_{-},\psi_{+}}\pm \varepsilon \phi$ with $\phi \in H^2_{*}(\Omega)$ and $\varepsilon>0$ sufficiently small, it follows that $u_{f,\psi_{-},\psi_{+}}$ satisfies the partially hinged plate problem \eqref{loadpb0weak}.

 Recalling that $H^2_{*}(\Omega) \subset C^0(\overline \Omega)$, to $u_{f,\psi_{-},\psi_{+}} $ we may associate the \emph{gap function}:
$$\mathcal{G}_{f,\psi_{-},\psi_{+}}(x):= u_{f,\psi_{-},\psi_{+}} (x,l)-u_{f,\psi_{-},\psi_{+}} (x,-l) \quad x \in[0,\pi]$$
and the \emph{maximal gap}:
\begin{equation}
    \mathcal{G}^\infty_{f,\psi_{-},\psi_{+}}:= \max_{x \in [0,\pi]} |\mathcal{G}_{f,\psi_{-},\psi_{+}}(x)|. \label{gap2bis}
\end{equation}
As already recalled in the introduction, the gap function (originally introduced in \cite{berchio} for the partially hinged plate problem without obstacles) measures the difference on the vertical displacements on the two free edges of the plate, therefore the maximal gap provides a measure of its torsional response. This is the quantity one aims to keep under control, and possibly to minimize, in order to improve the stability of the structure. We observe that if we remove the condition $|\psi_{\pm}|\geq\gamma_{\pm}$ in the definition of $\text{O}_{\pm}(\gamma_\pm)$, then $\psi_{\pm}\equiv 0$ becomes admissible and $u_{f,0,0}\equiv 0$ is trivially the unique minimizer of the energy. Therefore,  $\mathcal{G}_{f,0,0}\equiv 0 \equiv \mathcal{G}^\infty_{f,0,0}$.

We are now in a position to formally state our optimization problem. Let $\mathcal{F}$ and $\Psi_{\pm}$ be suitable subsets of, respectively, $(C^0(\overline{\Omega}))'$ and $ \text{O}_{\pm}(\gamma_\pm)$. Given $(\psi_{-},\psi_{+})  \in  \Psi_{-} \times \Psi_{+}$, we first look for the \textit{worst} $f \in \mathcal{F}$ yielding the maximum of the maximal gap:
\begin{equation}\label{max max2bis}
    \mathcal{G}^\infty_{\psi_{-},\psi_{+}}:= \max_{f \in \mathcal{F}} \mathcal{G}^\infty_{f,\psi_{-},\psi_{+}} = \max_{f \in \mathcal{F}} \max_{x \in [0,\pi]} |\mathcal{G}_{f,\psi_{-},\psi_{+}}(x)|,
\end{equation}
and then the \textit{best} obstacles $(\psi_{-},\psi_{+})  \in  \Psi_{-} \times \Psi_{+}$ minimizing the torsional response of the plate, namely such that
\begin{align}\label{min max max2bis}
    \mathcal{G}^\infty:=& \min_{(\psi_{-},\psi_{+})  \in  \Psi_{-} \times \Psi_{+}}\mathcal{G}^\infty_{\psi_{-},\psi_{+}}
    \,.
\end{align}

In Section \ref{ex2} we provide some classes $\mathcal{F}$ and  $\Psi_{\pm}$ in which \eqref{max max2bis} and \eqref{min max max2bis} admit a solution, while in Section \ref{Green2} we provide some qualitative information about them.

\subsection{Existence results}\label{ex2}

We start by studying the well-posedness of $\mathcal{G}^\infty_{\psi_{-},\psi_{+}}$.
\begin{thm}\label{Th max maxbis}
Let $\gamma_{-}, \gamma_{+}>0$ and $(\psi_{-},\psi_{+}) \in \text{O}_-(\gamma_{-}) \times \text{O}_+(\gamma_{+})$. 
The problem
\begin{equation}\label{Pb2bis}
    \mathcal{G}^\infty_{\psi_{-},\psi_{+}}:= \max_{f \in \mathcal{F}} \mathcal{G}^\infty_{f,\psi_{-},\psi_{+}} \quad \text{with} \quad  \mathcal{F}:=\left\{f \in 
     (C^0(\overline{\Omega}))' 
     \ : \ \Vert f \Vert_{(C^0(\overline{\Omega}))' 
     } \leq 1  \right\}
\end{equation}
admits a solution. Furthermore, if $-\psi_-\equiv \psi_+\equiv \psi$, and $f$ is a maximizer of $\mathcal{G}^\infty_{-\psi,\psi}$, then $-f$ is a maximizer as well.
\end{thm}

\medskip
Before presenting the proof of the above statement, we show that the map $ \mathcal{G}^\infty_{f,\psi_{-},\psi_{+}}$ defined in \eqref{gap2bis} is continuous with respect to its arguments.

\begin{prop}\label{contbis1}
Let $\text{O}_{\pm}(\gamma_{\pm})$ be the classes defined in \eqref{O+bis} and \eqref{O-bis}, with $\gamma_{-}, \gamma_{+}>0$. Then the map $$(f,\psi_{-},\psi_{+}) \in (C^0(\overline{\Omega}))' \times \text{O}_-(\gamma_{-}) \times \text{O}_+(\gamma_{+}) \mapsto \mathcal{G}^\infty_{f,\psi_{-},\psi_{+}} \in [0,+\infty) $$ is sequentially continuous, when $(C^0(\overline{\Omega}))'$ is endowed with the weak* topology and $\text{O}_{\pm}(\gamma_{\pm})$ are endowed with the sup-norm topology.
\end{prop}

\proof
Let $\{(f_n,\psi_{-,n}, \psi_{+,n}) \}_n \subset (C^0(\overline{\Omega}))' \times \text{O}_-(\gamma_{-}) \times \text{O}_+(\gamma_{+})$ be such that $(f_n,\psi_{-,n},\psi_{+,n}) \to (f,\psi_{-},\psi_{+})$ as $n \to + \infty$, that is $f_n \rightharpoonup^* f$ in $ (C^0(\overline{\Omega}))'$ 
and $\Vert \psi_{\pm,n} - \psi_{\pm} \Vert_{L^\infty(\Omega_{\text{O}})} \to 0$ as $n \to + \infty$. For simplicity of notation, we denote by $u=u_{f,\psi_{-},\psi_{+}}$ and $u_n=u_{f_n,\psi_{-,n},\psi_{+,n} }$, for all $n \in \mathbb{N}$, the corresponding solutions of \eqref{var1bis}. In particular, for every $n \in \mathbb{N}$, it holds
 \begin{equation}\label{eq1.2}
     (u_n,\varphi-u_n)_{H^2_{*}(\Omega)}\ge \langle f_n, \varphi-u_n \rangle \qquad \forall \varphi \in  H^2_{*,\psi_{-,n},\psi_{+,n}}(\Omega).
 \end{equation}
 Since the zero constant function belongs to the class $H^2_{*,\psi_{-,n},\psi_{+,n}}(\Omega)$, we have
 $$(u_n,-u_n)_{H^2_{*}(\Omega)}\ge \langle f_n, -u_n \rangle $$ and we conclude there exists $C>0$ independent of $u_n$ such that
 $$\|u_n\|_{H^2_*(\Omega)}\leq  C\|f_n\|_{(C^0(\overline{\Omega}))'}\,.$$
Then, the weak$^*-$convergence of $f_n$ to $f$ in $(C^0(\overline{\Omega}))'$ gives that $\{ u_n \}_n$ is bounded in $H^2_*(\Omega)$, and hence there exists $\overline{u} \in H^2_*(\Omega)$ such that, up to a subsequence, $u_n \rightharpoonup \overline{u}$ weakly in $H^2_*(\Omega)$. Moreover, since
$$\psi_{-,n} \le u_n \le \psi_{+,n} \quad \text{in} \  \Omega_{\text{O}}$$
and
$$\psi_{-,n} \to \psi_{-}, \quad \psi_{-,n} \to \psi_{+} \quad \text{pointwise in} \ \Omega_{\text{O}}, \quad u_n \to \overline{u} \quad \text{pointwise in} \ \Omega \ \text{as} \ n \to + \infty,$$
 we have that $\overline{u} \in H^2_{*,\psi_{-}, \psi_{+}}(\Omega)$.

We claim that $\overline{u}$ satisfies the variational inequality
\begin{equation}
    (\overline{u},\varphi-\overline{u})_{H^2_{*}(\Omega)}\ge \langle f, \varphi-\overline{u} \rangle \qquad \forall \varphi \in  H^2_{*,\psi_{-},\psi_{+}}(\Omega). \label{varinbis}
\end{equation}
Given $0<\varepsilon< \frac{\gamma_++\gamma_-}{2}$, the uniform convergence of $\psi_{-,n}$ to $\psi_{-}$ and of $\psi_{+,n}$ to $\psi_{+}$ yields that there exists $n_0 \in \mathbb{N}$ such that for all $n > n_0$
\begin{equation}
    \psi_{-,n}< \psi_{-}+\varepsilon < \psi_{+}-\varepsilon < \psi_{+,n} \quad \text{in} \ \Omega_{\text{O}}. \label{unifbis}
\end{equation}
Then, for $0<\varepsilon< \min\Big\{\frac{\gamma_++\gamma_-}{2}, \min\{\min|\psi_+|, \min|\psi_-|\}\Big\}$ and for $\varphi \in  H^2_{*,\psi_{-},\psi_{+}}(\Omega)$ fixed, we define $\varphi_\varepsilon \in  H^2_{*,\psi_{-},\psi_{+}}(\Omega)$ as follows:
$$\varphi_\varepsilon:= K_\varepsilon \varphi \quad \text{with} \quad K_\varepsilon:=1-\frac{\varepsilon}{\min\{\min|\psi_+|, \min|\psi_-|\}}\in(0,1)\,.$$
It is readily seen that $ \varphi_\varepsilon \to \varphi$ strongly in $H^2_{*}(\Omega)$ as $\varepsilon \to 0$. Furthermore, by construction we have
$$\varphi_\varepsilon \leq \psi_+-\varepsilon \frac{ |\psi_+|}{\min\{\min|\psi_+|, \min|\psi_-|\}}\leq \psi_+-\varepsilon $$
and
$$\varphi_\varepsilon \geq \psi_-+ \varepsilon\frac{ |\psi_-|}{\min\{\min|\psi_+|, \min|\psi_-|\}}\geq \psi_-+ \varepsilon\,.$$
Hence, thanks to \eqref{unifbis}, we obtain

$$\varphi_\varepsilon \in \bigcap_{n > n_0} H^2_{*,\psi_{-,n},\psi_{+,n}}(\Omega)\,,$$
therefore, it can be used as a test function in \eqref{eq1.2} for all $n >n_0$, that is
\begin{equation}
     (u_n,\varphi_\varepsilon-u_n)_{H^2_{*}(\Omega)}\ge \langle f_n, \varphi_\varepsilon-u_n \rangle \qquad \forall n>n_0. \label{varin2bis}
 \end{equation}

Recalling that $u_n \rightharpoonup \overline{u}$ in $H^2_{*}(\Omega)$, hence $u_n \to \overline{u}$ in $C^0(\overline{\Omega})$ (from the compactness of the embedding $H^2(\Omega)\subset C^0(\overline{\Omega})$), exploiting the weak lower semicontinuity of the norm, and recalling that $f_n \rightharpoonup f$ weakly* in $(C^0(\overline{\Omega}))'$, we pass to the limit as $n \to +\infty$ in \eqref{varin2bis} getting
 \begin{equation}
    (\overline{u},\varphi_\varepsilon-\overline{u})_{H^2_{*}(\Omega)}\ge \langle f, \varphi_\varepsilon-\overline{u} \rangle . \label{varin3}
 \end{equation}
 Then, passing to the limit as $\varepsilon \to 0$ (since $ \varphi_\varepsilon \to \varphi$ strongly in $H^2_{*}(\Omega)$), from the arbitrariness of $\varphi \in  H^2_{*,\psi_{-},\psi_{+}}(\Omega)$, \eqref{varin3} proves \eqref{varinbis}. Therefore, we can conclude that $u = \overline{u}$ in $\Omega$, by the uniqueness of solution.

Finally, from the fact that $u_n \to u$ in $C^0(\overline{\Omega})$, we conclude that $\mathcal{G}_{f_n,\psi_{-,n},\psi_{+,n}} \to \mathcal{G}_{f,\psi_{-},\psi_{+}}$ uniformly as $n \to + \infty$ in $[0,\pi]$, i.e.\ $\mathcal{G}^\infty_{f_n,\psi_{-,n},\psi_{+,n}} \to \mathcal{G}^\infty_{f,\psi_{-},\psi_{+}}$ as $n \to + \infty$.
\endproof

\proof [Proof of Theorem \ref{Th max maxbis}]
For $(\psi_{-},\psi_{+}) \in \text{O}_-(\gamma_{-}) \times \text{O}_+(\gamma_{+})$ fixed, let $\{ f_n \}_n \subset (C^0(\overline{\Omega}))'$
be a maximizing sequence for \eqref{Pb2bis} such that $\|f_n\|_{
  (C^0(\overline{\Omega}))'} \leq 1$. Since $\{ f_n \}_n$ is bounded in $ 
  (C^0(\overline{\Omega}))' 
 $, there exists $\overline{f} \in  
  (C^0(\overline{\Omega}))' 
 $ such that, up to a subsequence, $f_n \rightharpoonup \overline{f}$ in $
  (C^0(\overline{\Omega}))' $.  Then, by Proposition \ref{contbis1}, we get
$$ \mathcal{G}^\infty_{\overline{f},\psi_{-},\psi_{+}}=\mathcal{G}^\infty_{\psi_{-},\psi_{+}}.$$
 On the other hand, by weak lower semicontinuity of the norm, we have
$$ \Vert \overline{f} \Vert_{(C^0(\overline{\Omega}))'} \le \liminf_{n \to + \infty} \Vert f_n \Vert_{(C^0(\overline{\Omega}))'}\leq 1.$$
Hence, $\overline{f} \in \mathcal{F}$ and solves problem \eqref{Pb2bis}. Finally, the proof of the last part of the statement follows arguing as in the proof of Theorem \ref{Th max max}. 
\endproof

We now prove that $  \mathcal{G}^\infty$ is well defined if the minimum in \eqref{min max max2bis} is taken over a suitable class of obstacles $\Psi_{\pm} \subset \text{O}_{\pm}(\gamma_\pm)$ given in the following:
\begin{thm}\label{min exists} 
    Let $\kappa_{\pm}\geq \gamma_{\pm}>0$.   Then, the problem
\begin{align}\label{minpbbis}
    \mathcal{G}^\infty:= \min_{(\psi_{-},\psi_{+})  \in  \Psi_{-} \times \Psi_{+}}\mathcal{G}^\infty_{\psi_{-},\psi_{+}}
    \end{align}
    with $\Psi_{\pm}=\Psi_{\pm}(\gamma_{\pm}, k_{\pm}, \Omega_{\text{\emph {O}}})$ defined as follows
   \begin{equation}\label{PSI}
   \Psi_{\pm}:=\left\{ \psi_{\pm}\in \text{\emph{O}}_{\pm}(\gamma_{\pm})\cap C^{0,\alpha}(
   \Omega_{\text{\emph{O}}})  \text{ for some } 0<\alpha <1 \,:\, \|\psi_{\pm}\|_
   {C^{0,\alpha} (\Omega_{\text{\emph{O}}})
   }
   \leq \kappa_{\pm} 
   \right\}
   \end{equation}
    admits a solution.\footnote{Here, we used the customary notation $ $$\|v \|_{C^{0,\alpha} (\Omega_{\text{\emph{O}}})} := \Vert v \Vert_{L^\infty(\Omega_{\text{\emph{O}}})}+ \sup_{x,y \in \Omega_{\text{\emph{O}}}, x \neq y} \frac{|v(x)-v(y)|}{|x-y|^{\alpha}} $. } 
\end{thm}

\proof
Let $\{(\psi_{-,n},\psi_{+,n}) \}_n\subset \Psi_{-}\times \Psi_{+} $ be a minimizing sequence for problem \eqref{minpbbis}. Then, by Ascoli-Arzelà theorem, there exist $\overline{\psi}_{\pm}  \in \mathcal{C}^{0,\alpha}(\Omega_{\text{O}}) $ such that, up to a subsequence, $\psi_{\pm,n} \to \overline{\psi}_{\pm}$ in $\mathcal{C}^0(\Omega_{\text{O}})$. Moreover, it holds $\Vert \overline{\psi}_{\pm} \Vert_{\mathcal{C}^{0,\alpha}(\Omega_{\text{O}})} \le \kappa_{\pm} $, $\overline{\psi}_-\leq -\gamma_{-}$ and $\overline{\psi}_+\geq \gamma_{+}$. Hence, $\overline{\psi}_{\pm}  \in \Psi_{\pm} $. Furthermore, by Proposition \ref{contbis1}, arguing as in the proof of Theorem \ref{minthm}, it follows that the functional $ (\psi_-,\psi_+)\in \Psi_{-}\times \Psi_{+}\mapsto \mathcal{G}_{\psi_{-},\psi_{+}}^\infty$ is lower-semicontinuous with respect to the $C^0$-convergence. Then, the existence of a solution to problem \eqref{minpbbis} follows from the Direct Method in the Calculus of Variations.
\endproof

\subsection{Qualitative properties of worst forces and best obstacles}\label{Green2}
Theorems \ref{Th max maxbis} and \ref{min exists} provide suitable classes $\mathcal{F}$ and $\Psi_{\pm}$ such that problems \eqref{max max2bis} and \eqref{min max max2bis} admit a solution, namely under which worst forces and best obstacles exist. The aim of this section is to provide some qualitative information about them.

\subsubsection{Symmetry properties}

We first observe that the solutions to \eqref{var1bis} inherit the symmetry properties of the datum $f$ assuming suitable symmetry in the obstacle functions (and their domain $\Omega_{\text{O}}$).
To this aim we introduce the
subspaces of even and odd functions with respect to $y$:

$$C^0_\mathcal{E}(\overline{\Omega}):=\{u\in C^0(\overline{\Omega}): u(x,-y)=u(x,y)\quad\forall (x,y)\in \overline{\Omega}
\},$$

$$C^0_\mathcal{O}(\overline{\Omega}):=\{u\in C^0(\overline{\Omega}): u(x,-y)=-u(x,y)\quad\forall (x,y)\in \overline{\Omega} \}\,.$$
There holds
\begin{equation}\label{somma diretta}
C^0(\overline{\Omega})=C^0_\mathcal{E}(\overline{\Omega})\oplus C^0_\mathcal{O}(\overline{\Omega})
\end{equation}
and according to this decomposition we denote the components $u^{{e}}\in C^0_\mathcal{E}(\overline{\Omega})$ and $u^{{o}}\in C^0_\mathcal{O}(\overline{\Omega})$, i.e.
$$u^{{e}}(x,y)=\frac{u(x,y)+u(x,-y)}{2}, \quad u^{{o}}(x,y)=\frac{u(x,y)-u(x,-y)}{2}$$
and  the projections
$$\mathcal{P}_{\mathcal{E}}:u\in C^0(\overline{\Omega})\rightarrow u^{{e}}\in C^0_{\mathcal{E}}(\overline{\Omega})\,, \quad \mathcal{P}_{\mathcal{O}}:u\in C^0(\overline{\Omega})\rightarrow u^{{o}}\in C^0_{\mathcal{O}}(\overline{\Omega}).$$
We observe that
\begin{equation}\label{ortogonal}
\left(u^e,u^o\right)_{H^2_*(\Omega)}=0\quad \hbox{ for all } u\in H^2_*(\Omega)\,.
\end{equation}
Finally, we define
$$(C^0(\overline{\Omega}))'_{\mathcal{E}}=\{f\in (C^0(\overline{\Omega}))': \langle f,u \rangle=0 \quad \forall u\in C^0_\mathcal{O}(\overline{\Omega}) \}$$
and
$$(C^0(\overline{\Omega}))'_{\mathcal{O}}=\{f\in (C^0(\overline{\Omega}))': \langle f,u \rangle=0 \quad \forall u\in C^0_\mathcal{E}(\overline{\Omega}) \}.$$
We get that $(C^0(\overline{\Omega}))'=(C^0(\overline{\Omega}))'_{\mathcal{E}}\oplus (C^0(\overline{\Omega}))'_{\mathcal{O}}$ and for every $f\in (C^0(\overline{\Omega}))'$ we can write $f=f^{{e}}+f^{{o}}$, where $f^{{e}}:=f\circ \mathcal{P}_{\mathcal{E}}\in (C^0(\overline{\Omega}))'_{\mathcal{E}}$ and $f^{{o}}:=f\circ \mathcal{P}_{\mathcal{O}}\in (C^0(\overline{\Omega}))'_{\mathcal{O}}$.

As usual, we endow the dual space with the norm $\|f\|_{(C^0(\overline{\Omega}))'}=\sup_{\|v\|_{C^0(\overline{\Omega})}=1} |\langle f,v \rangle|$.
By the very definition of dual norm, we observe that the following inequality is satisfied:
\begin{equation}\label{noma e o}
\max\left\{\|f^e\|_{(C^0(\overline{\Omega}))'},\|f^o\|_{(C^0(\overline{\Omega}))'}\right\} \le \|f\|_{(C^0(\overline{\Omega}))'}\,.
\end{equation}


For the sake of notation, in the following, when $\psi_{+}\equiv -\psi_{-}\equiv \psi$, we shall write $H^2_{*,\psi}$,
$\mathcal{G}_{f,\psi}$, $\mathcal{G}^\infty_{f,\psi}$, $\mathcal{G}^\infty_{\psi}$ instead of $H^2_{*,-\psi, \psi}$, $\mathcal{G}_{f,-\psi,\psi}$, $\mathcal{G}^\infty_{f,-\psi,\psi}$, $\mathcal{G}^\infty_{-\psi,\psi}$, respectively.

\begin{lem}\label{lemsymbis}
    Let $f\in (C^0(\overline{\Omega}))'$, $\psi_{\pm} \in \text{O}_{\pm}(\gamma{_\pm})$ for some $\gamma{_\pm}>0$ and assume that $\Omega_{\text{\emph{O}}}$ is symmetric with respect to the $x$-axis, namely $(x,y)\in \Omega_{\text{\emph{O}}}$ if and only if $(x,-y)\in \Omega_{\text{\emph{O}}}$. Furthermore, let $u=u_{f,\psi_{-},\psi_{+}} \in H^2_{*,\psi_{-},\psi_{+}}(\Omega)$ satisfy the corresponding variational inequality \eqref{var1bis}. The following implications hold:
    \begin{itemize}
        \item[(i)] if $\psi_{-}$ and $\psi_{+}$ are even in $y$ and $f^o=0$, then $u^o \equiv 0$ in $\overline \Omega$;
        \item[(ii)] if $\psi_{-}=-\psi_{+}=\psi$, with $\psi$ even in $y$ and $f^e=0$, then $u^e\equiv 0$ in $\overline \Omega$.
    \end{itemize}
\end{lem}
\proof
We first notice that, by exploiting the definition of even and odd parts, one has
\begin{equation}\label{even_obst}
\varphi \in H^2_{*,\psi_{-},\psi_{+}}(\Omega) \Rightarrow 
\psi_{-}^e\leq \varphi^e \leq \psi_{+}^e \quad \text{ and } \quad \varphi \in H^2_{*,\psi}(\Omega) \Rightarrow -\psi^e \leq \varphi^e,\varphi^o \leq \psi^e   \,.
\end{equation}
Assume that $u \in H^2_{*,\psi_{-},\psi_{+}}(\Omega)$ is the solution to the variational inequality \eqref{var1bis}. We shall prove that if $f^o=0$, then $u^o=0$.
Using the decomposition  \eqref{somma diretta}, inequality \eqref{var1bis} can be written in the following way
\begin{align}
    &(u^e,\varphi^e-u^e)_{H^2_{*}(\Omega)} +(u^o,\varphi^o-u^o)_{H^2_{*}(\Omega)} \ge \langle f^e, \varphi^e-u^e \rangle +  \langle f^o, \varphi^o-u^o \rangle, \label{vardbis}
\end{align}
for every $\varphi \in H^2_{*, \psi_-,\psi_+}(\Omega)$.
 Now, since $\psi_{-}$ and $\psi_{+}$ are even in $y$, by the first implication in \eqref{even_obst}, we can test \eqref{vardbis} with $\varphi=u^e$, then using \eqref{ortogonal} and the assumption $f^o=0$, we get
$$\Vert u^o \Vert^2_{H^2_*(\Omega)} \le 0.$$
Hence, $u^o=0$.
The second assertion follows in a similar way by testing \eqref{vardbis} with $\varphi=u^o$, in view of  the second implication in \eqref{even_obst}.
\endproof

Next, we consider the symmetry with respect to the line $x=\frac{\pi}{2}$.
\begin{lem}
 Let $f\in L^p(\Omega)$ for some $p\in[1,+\infty]$, $\psi_{\pm} \in \text{O}_{\pm}(\gamma{_\pm})$ for some $\gamma{_\pm}>0$ and assume that $\Omega_{\text{\emph{O}}}$ is symmetric with respect to the line $x=\frac{\pi}{2}$. Furthermore, let $u=u_{f,\psi_{-},\psi_{+}} \in H^2_{*,\psi_{-},\psi_{+}}(\Omega)$ satisfy the corresponding variational inequality \eqref{var1bis}. If $\psi_{\pm}(x,y)=\psi_{\pm}(\pi-x,y)$ for all $(x,y) \in \Omega_{\text{\emph{O}}}$ and $f(x,y)=f(\pi-x,y)$ for all $(x,y) \in \Omega$, then  $$u(x,y)=u(\pi-x,y)\quad  \text{ for all } (x,y) \in \overline{\Omega}\,.$$
\end{lem}
\proof
If $f\in L^p(\Omega)$, inequality \eqref{var1bis} with $u=u_{f,\psi_{-},\psi_{+}} \in H^2_{*,\psi_{-},\psi_{+}}(\Omega)$ writes
 \begin{equation*}
     (u,\varphi-u)_{H^2_{*}(\Omega)} \ge \int_{\Omega}f(\varphi-u)\,dx dy \qquad \forall \varphi \in H^2_{*,\psi_{-},\psi_{+}}(\Omega)\,.
 \end{equation*}
 Set $w(x,y)=u(\pi-x,y)$ and $\zeta(x,y)=\varphi(\pi-x,y)$ , for all $(x,y) \in \Omega$. In view of the symmetry assumptions on $\psi_{\pm}$ and $\Omega_{\text{O}}$, we have $w,\zeta \in  H^2_{*,\psi_{-},\psi_{+}}(\Omega)$. Then, by changing the variable in the above inequality and recalling the symmetry assumption of $f$, we readily get
  \begin{equation*}
     (w,\zeta-u)_{H^2_{*}(\Omega)} \ge \int_{\Omega}f(\zeta-w)\,dx dy \qquad \forall \zeta \in H^2_{*,\psi_{-},\psi_{+}}(\Omega)\,.
 \end{equation*}
 By uniqueness, $u\equiv w$ in $\Omega$ and the proof is complete.
\endproof

We are in a position to prove that the worst force(s) $f$  (whose existence is ensured by Theorem \ref{Th max maxbis} in the set $\mathcal{F}$) must have a nontrivial odd part when the obstacle functions are even in $y$.
 
\begin{prop}\label{symprop}
Assume that $\Omega_{\text{\emph{O}}}$ is symmetric with respect to the $x$-axis and that $\psi_{+}\equiv -\psi_{-}\equiv \psi$ with $\psi\in \text{O}_{+}(\gamma)$ for some $\gamma>0$, even in $y$. If $f \in (C^0(\overline{\Omega}))' $ is the corresponding maximizer of $\mathcal{G}^\infty_{\psi}$ as defined in \eqref{Pb2bis}, then $f^o\neq 0.$ Moreover, if the contact sets $\Omega_-$ and $\Omega_+$ as defined in \eqref{sets} are both empty, then $f^o$ is a maximizer.
\end{prop}

\proof  For $\psi\in \text{O}_{+}(\gamma)$, even in $y$, fixed, let $f \in (C^0(\overline{\Omega}))'$ be the maximizer of $\mathcal{G}^\infty_{\psi}$ and let $u=u_{f,\psi} \in H^2_{*,\psi}(\Omega)$ be the corresponding solution to \eqref{var1bis}. By the definition of the gap function and the maximal gap, we have
\begin{equation}\label{oddgap}
\mathcal{G}_{f,\psi}(x)=u^o(x,l)-u^o(x,-l)
\qquad \text{and} \qquad
\mathcal{G}^\infty_{f,\psi}= \max_{x \in [0,\pi]} |u^o(x,l)-u^o(x,-l)|.
\end{equation}
If $f^o = 0$, by Lemma \ref{lemsymbis} point i), we have $u^o=0$, hence $\mathcal{G}^\infty_{f,\psi}=0$ and $f$ cannot be a maximizer for $\mathcal{G}^\infty_{\psi}$. Therefore, $f^o \neq 0$.
\par
For the last part of the proof, assume that the contact sets are both empty, then $u=u_{f,\psi}$ satisfies the partially hinged plate problem \eqref{loadpb0weak}. By testing \eqref{loadpb0weak} with odd test functions we get 
\begin{equation*}
	(u^o,\varphi^o)_{H^2_*(\Omega)} =\langle f^o, \varphi^o \rangle  \qquad\forall \varphi\in H^2_*(\Omega)\,,
	\end{equation*}
    and then
    \begin{equation*}
	(u^o,\varphi)_{H^2_*(\Omega)} =\langle f^o, \varphi \rangle  \qquad\forall \varphi\in H^2_*(\Omega)\,.
	\end{equation*}
Since $|u^o|<\psi$ in $\Omega$, this means that $u^o$ solves \eqref{var0} with $f=f^o$. In view of \eqref{oddgap} and since from \eqref{noma e o} $\|f^o\|_{(C^0(\overline{\Omega}))'} \leq 1$,  we infer that $\mathcal{G}^\infty_{\psi}=\mathcal{G}^\infty_{f,\psi}=\mathcal{G}^\infty_{f^o,\psi}$, namely $f^o$ is a maximizer.
\endproof

\vspace{0.5cm}

\subsubsection{Some remarks about the best obstacles under odd forces}\label{best}
Let $\mathcal{F}$ be the set of functions defined in Theorem \ref{Th max maxbis} and, for $\kappa_{ \pm} \geq \gamma_{\pm}>0$, let $\Psi_{ \pm}
=\Psi_{ \pm}
(\gamma_{ \pm}, \kappa_{ \pm}, \Omega_{\text{O}})$ be the set of obstacles defined in \eqref{PSI}. We denote $\Psi_{ \pm}^e=\Psi_{ \pm}^e(\gamma_{ \pm}, \kappa_{ \pm},\Omega_{\text{O}})$ the subset of $\Psi_{\pm}$ of even in $y$ obstacles (with domain $\Omega_{\text{O}}$ symmetric with respect to the $x$-axis). From Theorem \ref{Th max maxbis} and Proposition \ref{symprop} we know that, if $ \gamma_+\equiv \gamma_-\equiv \gamma$, $\kappa_+\equiv\kappa_-\equiv \kappa$ and $\psi_{+}\equiv -\psi_{-}\equiv \psi$ with $\psi \in\Psi_+^e=\Psi_+^e(\gamma,\kappa,\Omega_{\text{O}})$, then the corresponding {\it worst} force/s exists in $\mathcal{F}$ and must have a nontrivial odd part. Moreover, when the contact sets are empty, the worst force can be assumed to be odd. This motivates our choice in the following two sections to restrict the set of admissible worst forces to odd distributions, regardless of whether the contact sets are empty or not. Therefore, we consider the problem
$$\mathcal{G}^\infty_{\psi}= \max_{f \in \mathcal{F}, f\equiv f^o} \mathcal{G}^\infty_{f,\psi}= \mathcal{G}^\infty_{ f_{\psi},\psi}$$
for some $f_{\psi} \in (C^0(\overline{\Omega}))'_{\mathcal{O}}$. Then, we look for the {\it best} obstacles, namely minimizers of 
$$\mathcal{G}^\infty=\min_{\psi \in\Psi_+^e}  \mathcal{G}^\infty_{\psi}=\min_{\psi \in\Psi_+^e}  \mathcal{G}^\infty_{ f_{\psi},\psi}\,. $$
From Theorem \ref{min exists} the above problem admits a solution; we call an {\it optimal pair} any couple of worst force - best obstacle $(f_{\bar \psi},\bar \psi)\in (C^0(\overline{\Omega}))'_{\mathcal{O}} \times\Psi_+^e$ achieving $\mathcal{G}^\infty$, namely such that
\begin{equation}
\label{optimal}
\mathcal{G}^\infty= \mathcal{G}^\infty_{ f_{\bar\psi},\bar\psi}\,.
\end{equation}
In the following we provide some information about the value of $\mathcal{G}^\infty$ and optimal pairs which allow us to suggest, at the end of the section, the possible best obstacles. We start by remarking that, by statement $(ii)$ of Lemma \ref{lemsymbis}, the minimizer $u_{f,\psi}\in H^2_{*,\psi}(\Omega)$ of the energy $\mathbb{E}$, corresponding to any couple $(f, \psi)\in (C^0(\overline{\Omega}))'_{\mathcal{O}} \times\Psi_+^e$, is odd in $y$. Then, if we denote by $L_{\pm}$ the long edges of the plate, namely $L_-=[0,\pi]\times \{-l\}$ and $L_+=[0,\pi]\times \{l\}$, and we assume that $L_-\cup L_+ \subseteq \Omega_{\text{O}}$, we deduce that
$$ \mathcal{G}^\infty_{f,\psi}=  2\max_{x \in [0,\pi]} |u_{f,\psi}(x,l)|\leq 2 \max_{x \in [0,\pi]} \psi(x,l)$$
and, in turn, that
\begin{equation}
\label{upper_bound}         \mathcal{G}^\infty\leq 2 \min_{\psi \in\Psi_+^e }\max_{x \in [0,\pi]} \psi(x,l)=2\gamma. 
    \end{equation}
 The above minimum is trivially achieved by $\psi \equiv \gamma$ or, more in general by any $\psi\in\Psi_+^e$ satisfying $\psi \vert_{L_{+}}  \equiv \gamma$. However, the following proposition suggests that the upper bound $2\gamma$ for $\mathcal{G}^\infty$, given in \eqref{upper_bound}, might not be sharp.

\begin{thm}
\label{Prop min cost 1d}
    Assume that $\Omega_{\text{\emph{O}}}$ is symmetric with respect to the $x$-axis and that $L_-\cup L_+ \subseteq \Omega_{\text{\emph{O}}}$. Furthermore, let $(f_{\bar \psi},\bar \psi)\in (C^0(\overline{\Omega}))'_{\mathcal{O}} \times\Psi_+^e$ be an optimal pair as defined in \eqref{optimal} and denote by $u_{f_{\bar \psi},\bar \psi}\in H^2_{*, \bar\psi}(\Omega) $ the corresponding minimizer of the energy $\mathbb{E}$.  There holds
    \begin{itemize}
\item[(i)] $\mathcal{G}^\infty=\displaystyle{  \min_{\psi \in\Psi_+^e}  \mathcal{G}^\infty_{ f_{\psi},\psi}<2\gamma}$ if and only if $\|u_{f_{\bar \psi},\bar \psi}\|_{L^{\infty}(L_+)}< \gamma$;
\item[(ii)] $\mathcal{G}^\infty=\displaystyle{\min_{\psi \in\Psi_+^e}  \mathcal{G}^\infty_{ f_{\psi},\psi}}=2 \gamma$ if and only if $|u_{f_{\bar \psi},\bar \psi}(x_0,l)|=\gamma$ for some $x_0\in (0,\pi)$.
\end{itemize}
In particular, if case $(ii)$ occurs, then $\mathcal{G}^\infty=\displaystyle{\min_{\psi \in\Psi_+^e}  \mathcal{G}^\infty_{ f_{\psi},\psi}}=\mathcal{G}^\infty_{ f_{\psi_{\gamma}},\psi_{\gamma} }$ 
for any $\psi_{\gamma} \in\Psi_+^e:$ $\psi_{\gamma}\vert_{L_{+}}  \equiv \gamma$. 
\end{thm}

\proof
Let $(f_{\bar \psi},\bar \psi)\in (C^0(\overline{\Omega}))'_{\mathcal{O}} \times \Psi_+^e$ be an optimal pair, then the corresponding minimizer of the energy $\mathbb{E}$, $u_{f_{\bar \psi},\bar \psi} \in H^2_{*, \bar\psi}(\Omega)$, satisfies
$$\mathcal{G}^\infty =\mathcal{G}^\infty_{f_{\bar \psi},\bar \psi}=2\max_{x\in[0,\pi]}|u_{f_{\bar \psi},\bar \psi}(x,l)|=2\|u_{f_{\bar \psi},\bar \psi}\|_{L^{\infty}(L_+)}$$
and the statement $(i)$ follows at once. \par
As for the statement $(ii)$, if $\mathcal{G}^\infty=2 \gamma$, from the above formula we immediately infer that there must exist $x_0\in (0,\pi)$ such that $|u_{f_{\bar \psi},\bar \psi}(x_0,l)|=\gamma$. Conversely, if we assume that $|u_{f_{\bar \psi},\bar \psi}(x_0,l)|= \gamma$ for some $x_0\in (0,\pi)$. The fact that $(f_{\bar \psi},\bar \psi)\in (C^0(\overline{\Omega}))'_{\mathcal{O}} \times\Psi_+^e$ is an optimal pair readily gives that
$$\mathcal{G}^\infty=\mathcal{G}^\infty_{f_{\bar \psi},\bar \psi} \geq 2|u_{f_{\bar \psi},\bar \psi}(x_0,l)| = 2\gamma\,.$$
Whence, in view of \eqref{upper_bound} we immediately conclude that $\mathcal{G}^\infty=2\gamma$.
To complete the proof, let $u_{f_{\psi_{\gamma}},\psi_{\gamma}}\in H^2_{*, \psi_{\gamma}}(\Omega)$ be the minimizer of the energy corresponding to any obstacle $\psi_{\gamma}\in \Psi_+^e$ with $\psi_{\gamma} \in\Psi_+^e:$ $\psi_{\gamma}\vert_{L_{+}}  \equiv \gamma$ and to the worst among odd forces $f_{\psi_{\gamma}}\in(C^0(\overline{\Omega}))'_{\mathcal{O}}$. Namely,
$$\max_{f \in \mathcal{F}, f\equiv f^o} \mathcal{G}^\infty_{f,\psi_{\gamma}}= \mathcal{G}^\infty_{f_{\psi_{\gamma}},\psi_{\gamma}}=2\max_{x\in[0,\pi]}|u_{f_{\psi_{\gamma}},\psi_{\gamma}}(x,l)|\,.$$
Since $|u_{f_{\psi_{\gamma}}, \psi_{\gamma}}(x,l)|\leq \psi_{\gamma}\vert_{L_{+}}  \equiv \gamma$ for all $x\in[0,\pi]$, for what was proved above we deduce that 
$$\mathcal{G}^\infty=2\gamma \geq 2\max_{x\in[0,\pi]}|u_{f_{\psi_{\gamma}}, \psi_{\gamma}}(x,l)|=  \mathcal{G}^\infty_{f_{\psi_{\gamma}}, \psi_{\gamma}}\geq \displaystyle{\min_{\psi \in \Psi_+^e}  \mathcal{G}^\infty_{ f_{\psi},\psi}= \mathcal{G}^\infty}\,.$$
Namely, $(f_{\psi_{\gamma}},\psi_{\gamma})\in(C^0(\overline{\Omega}))'_{\mathcal{O}} \times\Psi_+^e$ is an optimal pair and the proof is complete.
\endproof

From Proposition \ref{Prop min cost 1d} we deduce that:
\begin{itemize} 
\item case $(i)$ can be seen as the safer case since $\mathcal{G}^\infty$ does not reach the upper bound \eqref{upper_bound} and $u_{f_{\bar \psi},\bar \psi}$ does not touch the obstacles along the long edges. In particular, if $\Omega_{\text{O}}=L_- \cup L_+$, then $u_{f_{\bar \psi},\bar \psi}$ solves the partially hinged plate problem \eqref{loadpb0weak} with $f=f_{\bar \psi}$;
\item if case $(ii)$ occurs, then the minimizer may touch the obstacles at the long edges, as it happens for the minimizer corresponding to the optimal couples $(f_{\psi_{\gamma}},\psi_{\gamma})$ defined in the last part of the statement. This is the situation in which obstacles help in improving the stability. Furthermore, the constant function $\psi\equiv \gamma$ is among the best obstacles.
\end{itemize} 
Based on the above observations we conclude that 
\begin{center}
\textit{a possible way to improve the torsional stability of the plate is\\ applying horizontal guides along its long edges (at levels $\pm \gamma$)\,.}
\end{center}

In the next section we provide sufficient conditions on $\gamma$ (in terms of the parameters of the plate) for case $(i)$ and $(ii)$ of Proposition \ref{Prop min cost 1d} to occur when $\mathcal{F}$ is properly chosen, see Theorems \ref{large delta} and \ref{small delta}.

\subsubsection{Best obstacles for antisymmetric delta-type forces} 
As in Section \ref{best} we assume that $ \gamma_+\equiv \gamma_-\equiv \gamma$, $\kappa_+\equiv\kappa_-\equiv \kappa$ and $\psi_{+}\equiv -\psi_{-}\equiv \psi$ with $\psi \in\Psi_+^e=\Psi_+^e(\gamma,\kappa,\Omega_{\text{O}})$ with $\Omega_{\text{O}}$ symmetric with respect to the $x$-axis, and we restrict our attention on odd forces. More precisely, here we focus on the odd distributions:
$$T_{\xi,\eta}:= \frac{\delta_{(\xi,\eta)}-\delta_{(\xi,-\eta)}}{2} \quad \text{with }(\xi,\eta)\in \overline \Omega$$
where $\delta_p$ is the Dirac delta with mass concentrated at $p\in \overline \Omega$. Clearly, $T_{\xi,\eta}\in (C^0(\overline{\Omega}))'_{\mathcal{O}}$ and $\|T_{\xi,\eta}\|_{(C^0(\overline{\Omega}))' } = 1$. The above choice is motivated by \cite[Section 4]{BF2}, where it was shown that the maximizing sequences found numerically for problem \eqref{gap2bis}, in the obstacle free case with $L^1$ loads, even if does not converge,
exhibits spikes with opposite signs in the boundary points that suggest a weak*-convergence to deltas
concentrated in these points. Then, we denote by $w_{\xi,\eta}\in H^2_{*,\psi}(\Omega)\subset C^0(\overline{\Omega})$ the minimizer of $\mathbb{E}$ corresponding to $f=T_{\xi,\eta}$. Hence, $w_{\xi,\eta}$ satisfies the variational inequality
 \begin{equation}
 \label{var00}
     (w_{\xi,\eta},\varphi-w_{\xi,\eta})_{H^2_{*}(\Omega)}  \ge \langle T_{\xi,\eta}, \varphi-w_{\xi,\eta} \rangle \qquad \forall \varphi \in H^2_{*,\psi}(\Omega).
 \end{equation}

Furthermore, we denote by $v_{\xi,\eta}\in H^2_*(\Omega)\subset C^0(\overline{\Omega})$ the solution to the partially hinged plate problem, namely
\begin{equation}\label{weakxieta}
(v_{\xi,\eta},\phi)_{H^2_*}=\langle T_{\xi,\eta},\phi\rangle\qquad\forall \phi\in H^2_*(\Omega)\,.
\end{equation}
 When the contact sets \eqref{sets} are empty, then $v_{\xi,\eta}$ and $w_{\xi,\eta}$ coincide. From the Fourier expansion of the Green function \eqref{green3}, one may deduce that of $v_{\xi,\eta}$. More precisely, there holds:

\begin{prop}\label{monotonia} \cite[Theorem 2.1]{BF2}
	Let $p=(\xi,\eta)\in \overline \Omega$. Then,
	\begin{equation*}
	v_{\xi,\eta}(x,y)=\dfrac{1}{4\pi}\sum_{m=1}^{+\infty} \dfrac{\phi_m(y,\eta)-\phi_m(y,-\eta)}{m^3}\sin(m\xi)\,\sin(mx) \qquad \forall (x,y)\in\overline \Omega\,,
	\end{equation*}
where the functions $\phi_m(y,\eta)$ are given explicitly in formula \eqref{soluzdirac2} and the series converges uniformly in $\overline{\Omega}$. 
	\end{prop}

 For a given $\psi\in\Psi_+^e$, we denote the gap function associated with $w_{\xi,\eta}$ and $v_{\xi,\eta}$, respectively, as$\mathcal{G}_{T_{\xi,\eta},\psi}$ and $\mathcal{G}_{T_{\xi,\eta}}$ with their maximal gaps: $\mathcal{G}_{T_{\xi,\eta}, \psi}^{\infty}$ 
 and $\mathcal{G}_{T_{\xi,\eta}}^{\infty}$. More precisely, since both $w_{\xi,\eta}$ and $v_{\xi,\eta}$ are odd in $y$, we have
\begin{align}\label{funzionale}
\mathcal{G}_{T_{\xi,\eta}, \psi}(x) &=2 w_{\xi,\eta}(x,l)\quad  \text{ and }  \quad \mathcal{G}_{T_{\xi,\eta}, \psi}^{\infty}= \max_{x\in[0,\pi]}\, |\mathcal{G}_{T_{\xi,\eta}}(x)|\, ,\\ \notag
\mathcal{G}_{T_{\xi,\eta}}(x)&=2 v_{\xi,\eta}(x,l)\quad \text{ and }  \quad\mathcal{G}_{T_{\xi,\eta}}^{\infty}= \max_{x\in[0,\pi]}\, |\mathcal{G}_{T_{\xi,\eta}}(x)|\, .
\end{align}
As usual, we first seek the \textit{worst} among the forces $T_{\xi,\eta}$ as $(\xi,\eta)$ vary in $\overline \Omega$, namely, the one maximizing the maximal gaps. In the obstacle-free case, from Proposition \ref{monotonia} it readily follows the Fourier expansion of the gap function $\mathcal{G}_{T_{\xi,\eta}}$. However, the complexity of the analytic expression of the coefficients $\phi_m$ makes it hard to determine the maximum points of the map $(\xi,\eta)\in \overline \Omega \mapsto \mathcal{G}_{T_{\xi,\eta}}^{\infty}$. We refer to \cite[Conjecture 5]{antunes} for a numerical solution of the problem in the obstacle-free case. An analytical proof was given in \cite[Theorem 2.3]{BF1}, by replacing $\overline \Omega$ with a suitable subset $\tilde \Omega$:
\begin{prop}\label{conj} \cite[Theorem 2.3]{BF1}
There holds
\begin{equation}\label{GGD}
\max_{(\xi,\eta)\in \tilde \Omega}\,  \mathcal{G}_{T_{\xi,\eta}}^{\infty}= \mathcal{G}^\infty_{T_{\frac{\pi}{2},\pm l}}\,,
\end{equation}
where
\begin{equation}\label{omega}
\tilde \Omega:=\left([0,z_0]\cup [\pi-z_0,\pi]\cup\left\{\frac{\pi}{2}\right\} \right)\times[-l,l] \cup[0,\pi]\times[-w_0,w_0]  \subset \overline \Omega
\end{equation}
with $0<z_0<\frac{\pi}{2}$ and $0<w_0<l$ explicitly given in \cite[Section 5]{BF1}.
\end{prop}

The goal of this section is to provide an explicit threshold for $\gamma$ so that the two cases of Proposition \ref{Prop min cost 1d} occur when restricting to the family of distributions considered in Proposition \ref{conj}, namely to the class of forces
$$\tilde{\mathcal F}:=\{T_{\xi,\eta}: (\xi,\eta)\in \tilde \Omega \}$$
with $\tilde \Omega$ as given in \eqref{omega}. Let $\Omega_{\text{O}}$ be a closed subset of $\overline \Omega$, symmetric with respect to the $x$-axis, and let $v_{\xi,\eta}$ be as given in Proposition \ref{monotonia}, we set
\begin{equation}\label{M}
M_{\xi, \eta}=M_{\xi, \eta}(\sigma, l, \Omega_{\text{O}}):=\max_{(x, y)\in \Omega_{\text{O}}} |v_{\xi,\eta}(x,y)| \quad \text{and} \quad M=M(\sigma, l,\Omega_{\text{O}}):=\max_{(\xi, \eta)\in \tilde \Omega} M_{\xi, \eta}\,.
\end{equation}

Clearly, $M_{0, \eta}=M_{\pi, \eta}=M_{\xi,0}=0$. Moreover, the map $(\xi, \eta)\in \tilde \Omega \mapsto M_{\xi, \eta}\in [0,M]$ is continuous.  In the following statement we prove that, when the forces lie in the set $\tilde{\mathcal F}$ and $\gamma>M$, then the obstacles cannot be exploited to improve the torsional stability of the plate:
\begin{thm}\label{large delta}
   Let $\Omega_{\text{\emph{O}}}$ be a not empty, closed subset of $\overline \Omega$, symmetric with respect to the $x$-axis and let $M=M(\sigma, l, \Omega_{\text{\emph{O}}})$ be as defined in \eqref{M}. Furthermore, consider the set $\Psi_+^e=\Psi_+^e(\gamma, \kappa,\Omega_{\text{\emph{O}}})$ with 
     \begin{equation}\label{delta_large}
    \kappa  >\gamma>M\,.
   \end{equation}
   Then $v_{\xi,\eta}$ is the unique minimizer of the functional $\mathbb{E}$ with $f=T_{\xi,\eta}\in\tilde{\mathcal F}$,  over the set
    $H^2_{*,\psi}(\Omega)$, namely it satisfies \eqref{var00}. Moreover, $T_{\pi/2, \pm l}$ is the worst force for all $\psi \in\Psi_+^e$, namely
   \begin{equation}\label{worst-expl}
     \max_{f \in \tilde{\mathcal F}} \mathcal{G}^\infty_{f,\psi} = \max_{f \in \tilde{\mathcal F}}  \max_{x \in [0,\pi]} |\mathcal{G}_{f,\psi}(x)|=\mathcal{G}^\infty_{T_{\frac{\pi}{2},\pm l}}\,.
\end{equation}
 In particular, 
   \begin{equation}\label{G tilde}
 \tilde{\mathcal{G}}^\infty:=\min_{\psi \in\Psi_+^e}\max_{f \in \tilde{\mathcal F}} \mathcal{G}^\infty_{f,\psi} =\mathcal{G}^\infty_{T_{\frac{\pi}{2},\pm l}}  \quad  \text{and} \quad \tilde{\mathcal{G}}^\infty<2\gamma\,.
 \end{equation}
\end{thm}

\begin{proof}
   Under the given assumptions, by recalling the definition of $M$, the solution $v_{\xi,\eta}$ of \eqref{weakxieta} satisfies $|v_{\xi,\eta}(x,y)|\leq M <\gamma \leq  \psi(x,y)$ for all $(x,y) \in \Omega_{\text{O}}$, all $(\xi, \eta) \in \tilde \Omega$ and all $\psi \in\Psi_+^e$. Hence, the contact sets of $v_{\xi,\eta}$ are empty, whence $w_{\xi,\eta}\equiv v_{\xi,\eta}$. Then, \eqref{worst-expl} readily comes from \eqref{GGD}. The second part of the statement instead follows from \eqref{worst-expl} and by noticing that $\tilde{\mathcal{G}}^\infty=\mathcal{G}^\infty_{T_{\frac{\pi}{2},\pm l}}=2 \displaystyle{\max_{x\in [0,\pi]}|v_{\frac{\pi}{2},\pm l}(x, l)|}\leq 2M <2 \gamma$.
\end{proof}

An explicit upper bound for $M$ and, in turn, an explicit lower bound for $\gamma$ so that condition \eqref{delta_large} holds, is given in Proposition \ref{upper} below. The case $\gamma\leq M$ is instead considered in the following statement where we show that $\psi\equiv \gamma$ is a best obstacle when $\Omega_{\text{\emph{O}}}=L_+ \cup L_-$, i.e., it improves the torsional stability. Furthermore, in this case $M$ is given explicitly.

\begin{thm} \label{small delta}
   Assume that $\Omega_{\text{\emph{O}}}=L_+ \cup L_-$ (thin obstacle problem)  and let $M=M(\sigma, l, \Omega_{\text{\emph{O}}})$ be as defined in \eqref{M}. Then
\begin{equation}\label{explicit M}
   M=M(\sigma, l)= \dfrac{4}{\pi}\sum_{m=1,\text{ odd}}^{+\infty} \frac{\sinh(m\ell)^2}{m^3(1-\sigma) [(3+\sigma)\sinh(2m\ell)+2m\ell (1-\sigma)]} \,. 
   \end{equation}
    Furthermore, consider the set $\Psi_+^e=\Psi_+^e(\gamma, \kappa,\Omega_{\text{\emph{O}}})$ with
   $$0<\gamma\leq M \quad \text{and } \kappa>\gamma$$ 
and let $\tilde{\mathcal{G}}^\infty$ be as defined in \eqref{G tilde}. There holds
    $$\tilde{\mathcal{G}}^\infty= 2\gamma \quad \text{and} \quad  \tilde{\mathcal{G}}^\infty= \mathcal{G}^\infty_{T_{\overline \xi,\overline \eta}, \gamma} \text{ for some }(\overline \xi,\overline \eta) \in \tilde \Omega\,.$$
     Namely, $(T_{\overline \xi,\overline \eta}, \gamma)\in  \tilde{\mathcal F}\times\Psi_+^e$ is an optimal pair for $\tilde{\mathcal{G}}^\infty$.
\end{thm}
\begin{proof}
We start by noticing that, by slightly modifying the proof of Theorems \ref{Th max maxbis} and \ref{min exists}, there exists an optimal pair as defined in \eqref{optimal}, say $(f_{\bar \psi},\bar \psi)\in \tilde{\mathcal{F}} \times\Psi_+^e$, such that
\begin{equation} \label{lower}
\tilde{\mathcal{G}}^\infty = \min_{\psi \in \Psi_+^e} \max_{f \in \tilde{\mathcal{F}}} \mathcal{G}^\infty_{f,\psi} = \mathcal{G}^\infty_{f_{\bar\psi}, \bar\psi} \geq \mathcal{G}^\infty_{f, \bar\psi} \quad \text{for all } f \in \tilde{\mathcal{F}}\,.
\end{equation}

On the other hand, since $ \Omega_{\text{O}} = L_+ \cup L_- $, \eqref{funzionale} combined with \eqref{M} yields $ M_{\xi, \eta} = \frac{1}{2} \mathcal{G}^\infty_{T_{\xi,\eta}} $ for all $ (\xi, \eta) \in \tilde{\Omega} $, while \eqref{GGD} gives $ M = \frac{1}{2} \mathcal{G}^\infty_{T_{\frac{\pi}{2}, \pm l}} $. Then, \eqref{explicit M} follows from \cite[Lemma 5.4]{BF1} where the latter value was computed.

Moreover, note that $ \left(\frac{\pi}{2}, \eta\right) \in \tilde{\Omega} $ for all $ \eta \in [-l, l] $, so $ T_{\frac{\pi}{2}, \eta} \in \tilde{\mathcal{F}} $ for all such $\eta $. The map
\[
\eta \in [-l, l] \mapsto \mathcal{G}^\infty_{T_{\frac{\pi}{2}, \eta}}
\]
is continuous and, by \cite[Proposition 5.1]{BF1}, it is even and strictly increasing in $ [0, l] $. In particular, it is bijective from $ [0, l] $ onto $ [0, 2M] $.
 Therefore, since $0<\gamma \leq M$, there exists $ \bar \eta = \bar \eta(\gamma) \in (0, l] $ such that $\mathcal{G}^\infty_{T_{\frac{\pi}{2}, \bar \eta}} = 2\gamma $, and for all $ 0 < \varepsilon < 2\gamma $, there exists $ \eta_\varepsilon \in (0, \bar \eta) $ such that $ \mathcal{G}^\infty_{T_{\frac{\pi}{2}, \eta_\varepsilon}} = 2\gamma - \varepsilon $.
Recalling the definition of $\mathcal{G}^\infty_{T_{\frac{\pi}{2}, \eta_\varepsilon}} $, this implies that
\[
|v_{\frac{\pi}{2}, \eta_\varepsilon}(x, y)| < \gamma \leq \psi(x, y) \quad \text{for all } (x, y) \in \Omega_{\text{O}} \text{ and all } \psi \in\Psi_+^e\,.
\]
 Hence, $v_{\frac{\pi}{2}, \eta_\varepsilon}$ solves the obstacle problem \eqref{var00} for all such $\psi$, hence $v_{\frac{\pi}{2}, \eta_\varepsilon}\equiv w_{\frac{\pi}{2}, \eta_\varepsilon}$. Then, by testing \eqref{lower} with $f=T_{\frac{\pi}{2}, \eta_\varepsilon}$, we get: 
\[
\tilde{\mathcal{G}}^\infty \geq \mathcal{G}^\infty_{T_{\frac{\pi}{2}, \bar \eta}, \bar\psi} = \mathcal{G}^\infty_{T_{\frac{\pi}{2}, \bar \eta}} = 2\gamma - \varepsilon\,.
\]
 Eventually, letting $\varepsilon \to 0^+$ and recalling the upper bound \eqref{upper_bound}, we get the value of $\tilde{\mathcal{G}}^\infty$. The last part of the thesis follows arguing as in the last part of the proof of Proposition \ref{Prop min cost 1d}.
\end{proof}
\nc

By combining the statements of Theorems \ref{large delta} and \ref{small delta}, we deduce the following:
\begin{cor}
Assume that $\Omega_{\text{\emph{O}}}=L_+ \cup L_-$ and let $M=M(\sigma, l)$ be as defined in \eqref{explicit M}. Furthermore, consider the set $\Psi_+^e=\Psi_+^e(\gamma, \kappa,\Omega_{\text{\emph{O}}})$ with $\kappa>\gamma>0$ and let $\tilde{\mathcal{G}}^\infty$ be as defined in \eqref{G tilde}. Then, we have
\begin{itemize}
    \item[(i)] $\gamma >M \Longleftrightarrow \tilde{\mathcal{G}}^\infty < 2 \gamma;$
    \item[(ii)] $\gamma  \le M \Longleftrightarrow \tilde{\mathcal{G}}^\infty = 2 \gamma.$
\end{itemize}     
\end{cor}

We conclude with an estimate that provides a more explicit indication of how much large $\gamma$ must be in Theorem \ref{large delta}. 

\begin{prop}\label{upper}
 Let $M=M(\sigma, l, \Omega_{\text{\emph{O}}})$ be as defined in \eqref{M}, there holds
\begin{equation}\label{C}
M\leq C(\sigma, l):= \frac{\pi \cosh^2(l) \left[5+2\sigma+\sigma^2+2l(5+2\sigma)(1-\sigma)+ 8l^2(1-\sigma)^2\right]}{6(1-\sigma)[(3+\sigma)\sinh(2l)-l(1+\sigma)]} +\frac{\pi}{12}\,.
\end{equation}
\end{prop}

\begin{proof}
In view of \eqref{decreasing},  since $\displaystyle{\sum_{m=1}^{+\infty} \dfrac{1}{m^3} \leq \sum_{m=1}^{+\infty} \dfrac{1}{m^2}=\dfrac{\pi^2}{6}}$, we immediately get that

\begin{equation*}
	|v_{\xi,\eta}(x,y)|\leq \left(\dfrac{1}{4\pi}\sum_{m=1}^{+\infty} \dfrac{1}{m^3} \right)
    \left(\phi_1(y,\eta)+\phi_1(y,-\eta) \right) \leq \frac{\pi}{24}\left(\phi_1(y,\eta)+\phi_1(y,-\eta) \right)  \qquad
	\end{equation*}
   $\forall (\xi,\eta)\in\tilde \Omega\,, \forall (x,y)\in\overline \Omega\,$.
    Furthermore, from  \eqref{soluzdirac2} we have that
    \begin{equation*}
		\begin{split}
		\phi_1(y,\eta)+\phi_1(y,-\eta)=
		e^{-l}\bigg[& 2 \cosh(\eta)\bigg(\dfrac{\overline \zeta(y,l)}{F(l)}+l\dfrac{\overline \psi(y,l) }{F(l)}
        \bigg)-2 \sinh( \eta)\eta\dfrac{\overline\psi(y,l)}{F(l)}\bigg]
		\\+&(1+|y-\eta|)e^{-|y-\eta|}+(1+|y+\eta|)e^{-|y+\eta|}
		\end{split}
		\end{equation*}
            Then, exploiting the estimates given in \eqref{coeff-est}, since $F(l)>0$, we get
        \begin{equation*}
		\begin{split}
		&\phi_1(y,\eta)+\phi_1(y,-\eta)\leq
		 \dfrac{2 e^{-l} \cosh(l) }{F(l)} \bigg[|\overline \zeta(y,l)|+2l |\overline \psi(y,l)| \bigg] +2\\
         & \leq\frac{4 \cosh^2(l) \left[5+2\sigma+\sigma^2+4l(3+\sigma)(1-\sigma)+ 8l^2(1-\sigma)^2\right]}{(1-\sigma)(3+\sigma)\sinh(2l)-l(1-\sigma^2)} +2 
         \,.
		\end{split}
		\end{equation*}
        Summing up, we have
$$
	|v_{\xi,\eta}(x,y)|\leq C(\sigma, l)  \qquad \forall (\xi,\eta)\in\tilde \Omega\,, \forall (x,y)\in\overline \Omega\,
    $$
   with $C(\sigma, l)$ as defined in \eqref{C}.
\end{proof}

\par
\smallskip
\noindent
\textbf{Acknowledgments.} The authors are members of the Gruppo Nazionale per l'Analisi Matematica, la Probabilit\`a e le loro Applicazioni (GNAMPA, Italy) of the Istituto Nazionale di Alta Matematica (INdAM, Italy). A.G. Grimaldi and F. Feo have been partially supported through the INdAM - GNAMPA 2025 Project "Regolarità di soluzioni di equazioni paraboliche a crescita non standard degeneri" (CUP: E5324001950001) and "Esistenza, unicità, simmetria e stabilità per problemi ellittici nonlineari e nonlocali" (CUP: E5324001950001), respectively. The research was partially carried out within the PRIN 2022 project: Geometric-Analytic Methods for PDEs and Applications GAMPA, ref. 2022SLTHCE (CUP: E53D2300588 0006 and I53D2300242 0006) funded by European Union - Next Generation EU within the PRIN 2022 program (D.D. 104 - 02/02/2022 Ministero dell'Universit\`a e della Ricerca). This manuscript reflects only the authors' views and opinions and the Ministry cannot be considered responsible for them.

\end{document}